\documentclass{scrartcl}

\usepackage[utf8]{inputenc}
\usepackage[english]{babel}
\usepackage{amsmath,amssymb}
\usepackage{amsthm,enumerate}
\usepackage[all]{xy}
\usepackage{multirow}
\usepackage{authblk}
\usepackage{mathrsfs}
\usepackage{hyperref}
\usepackage{microtype}
\usepackage{enumitem}
\setlist[enumerate,1]{label=\textup{(\alph*)}}

\newtheorem{theorem}{Theorem}[section]
\newtheorem{proposition}[theorem]{Proposition}
\newtheorem{lemma}[theorem]{Lemma}
\newtheorem{corollary}[theorem]{Corollary}

\theoremstyle{definition}
\newtheorem{definition}[theorem]{Definition}
\newtheorem{remark}[theorem]{Remark}{}
{}
\newtheorem{notation}[theorem]{Notation}{}

\newcommand{\F}{\mathbb{F}}
\newcommand{\Z}{\mathbb{Z}}
\newcommand{\gen}[1]{\langle #1 \rangle}
\newcommand{\FF}{\mathcal{F}}
\newcommand{\EE}{\mathcal{E}}
\newcommand{\LL}{\mathcal{L}}
\newcommand{\TT}{\mathcal{T}}
\newcommand{\foc}{\mathfrak{foc}}
\newcommand{\g}{\mathcal{G}}
\newcommand{\pp}{\mathcal{P}}
\newcommand{\hh}{\mathcal{H}}

\newcommand{\bb}{\mathcal{B}}
\newcommand{\definicio}{\stackrel{\text{def}} = }
\newcommand{\setp}{\mathfrak{X}}
\newcommand{\functor}{(-)^{\bullet}}

\newcommand{\GL}{\operatorname{GL}\nolimits}
\newcommand{\SL}{\operatorname{SL}\nolimits}
\newcommand{\Id}{\operatorname{Id}\nolimits}
\newcommand{\Hom}{\operatorname{Hom}\nolimits}
\newcommand{\Inn}{\operatorname{Inn}\nolimits}
\newcommand{\Aut}{\operatorname{Aut}\nolimits}
\newcommand{\Out}{\operatorname{Out}\nolimits}
\newcommand{\Syl}{\operatorname{Syl}\nolimits}
\newcommand{\Iso}{\operatorname{Iso}\nolimits}
\newcommand{\incl}{\operatorname{incl}\nolimits}
\newcommand{\Mor}{\operatorname{Mor}\nolimits}
\newcommand{\Map}{\operatorname{Map}\nolimits}
\newcommand{\Ob}{\operatorname{Ob}\nolimits}
\newcommand{\Ker}{\operatorname{Ker}\nolimits}
\newcommand{\rk}{\operatorname{rk}\nolimits}

\numberwithin{equation}{section}

\title{Some new examples of simple $p$-local compact groups}
\author{
	Alex Gonz\'alez\\
	\texttt{agondem@gmail.com}
	\and
	Toni Lozano\\
	\texttt{tonilb@mat.uab.cat}\footnote{Second author has been supported by MICINN grant BES-2011-044403.}
	\and
	Albert Ruiz\\
	\texttt{albert@mat.uab.cat}\footnote{All authors are partially supported by MICINN-FEDER project number MTM2016-80439-P.}
}
\date{}

\begin{document}

\maketitle

\begin{abstract}
In this paper we present new examples of simple $p$-local compact groups for all odd primes.
We also develop the necessary tools to show saturation, simpleness and the non-realizability as $p$-compact groups or compact Lie groups, which can be applied in a more general framework.
 \end{abstract}

\section{Introduction}
In \cite{BLO3}, C. Broto, R. Levi and B. Oliver defined the concept of $p$-local compact group: given a prime number $p$, a $p$-local compact group is a triple $(S,\FF,\LL)$ where $S$ is a discrete $p$-toral group, $\FF$ a saturated fusion system over $S$ and $\LL$ a centric linking associated to $\FF$. In \cite{BLO3}, the authors also prove that compact Lie groups and $p$-compact groups provide examples of $p$-local compact groups: given $G$ a compact Lie group (respectively a $p$-compact group $X$), there is a $p$-local compact group structure $(S,\FF,\LL)$ together with an inclusion $S\leq G$ (respectively a map $Bf\colon BS \to BX$) which is a Sylow $p$-subgroup, such that $\FF$ is the fusion system over $S$ induced by $G$ (respectively by $X$) and $\LL$ is a centric linking system associated to $\FF$. Moreover, the same authors also prove that, in these cases, $|\LL|^\wedge_p \simeq BG^\wedge_p$ (respectively $|\LL|^\wedge_p \simeq BX$).

We are interested in proving the existence of simple $p$-local compact groups (see Definition \ref{defisimple}) which do not correspond neither to compact Lie groups, nor to $p$-compact groups. The way we use to get examples of this kind is looking to the already known exotic $p$-local finite groups, identifying an abelian subgroup which plays the role of a torus, and checking if there is a way to consider a $p$-local compact group with this structure. An example with this property is the $3$-local finite group labelled as $\FF(3^{2k+1},3)$ in \cite[Table 6]{DRV}. This example gives a possible structure of a $3$-local finite group and, moreover, we have been able to fit this case in a family which can be defined for every odd prime $p$.

To state the main result of the paper we need some notation: given a prime $p$, let $C_p$ be the cyclic group of order $p$ (in multiplicative notation), and let $\Z/p^{\infty}$ denote the union of all the cyclic groups $\Z/p^n$ (in additive notation).

\begin{theorem}
	Let $p$ be an odd prime number. Consider the action of $C_p$ on $T\definicio(\Z/p^{\infty})^{p-1}$ given by matrix $B$ in Equation~\eqref{action_B}, and define the split extension $S\definicio(\Z/p^\infty)^{p-1}\rtimes C_p$. Consider $s$ an element of order $p$ in $S\setminus T$, $\zeta$ a generator of the center of $S$ and define $V\definicio\langle s, \zeta\rangle \cong C_p\times C_p$. Then, there exist $p$-local compact groups $(S,\FF,\LL)$ for each prime number $p\geq 3$,  and $(S,\widetilde{\FF},\widetilde{\LL})$ for $p\geq 5$ fulfilling the following table:
	{\renewcommand{\arraystretch}{2}
		\[
		\begin{array}{|c||c|c|c|c|}
		\hline \FF & \Aut_\FF(S) & \Aut_\FF(T) & \Aut_\FF(V) & \text{prime} \\
		\hline \hline \FF & \langle \phi, \psi, \Inn(S) \rangle  & \GL_2(\F_3) & \GL_2(\F_3) & p=3 \\
		\hline \FF &  \langle \phi^2, \psi\phi^{-1}, \Inn(S) \rangle & A_p \rtimes C_{p-1} & \SL_2(\F_p) \rtimes C_{(p-1)/2} & \multirow{2}{30pt}{$p\geq5$} \\
		\cline{1-4} \widetilde{\FF} &  \langle \phi, \psi, \Inn(S) \rangle & \Sigma_p \times C_{p-1} & \GL_2(\F_p) & \\ \hline
		\end{array}
		\]
	}
	and satisfying:
	\begin{enumerate}
		\item Neither $(S,\FF,\LL)$, nor $(S,\widetilde{\FF},\widetilde{\LL})$ can be realized by a compact Lie group, or by a $p$-compact group.
		\item For $p\geq 3$, the $p$-local compact groups $(S,\FF,\LL)$ are simple, and for $p\geq5$, $(S,\FF,\LL)$ is the only proper normal 	subsystem of $(S,\widetilde{\FF},\widetilde{\LL})$.
		\item The $p$-completed nerves of $\LL$ (for $p \geq 3$) and $\widetilde{\LL}$ (for $p \geq 5$) are simply connected. 
	\end{enumerate}
\end{theorem}
\begin{proof}
The fusion systems $(S,\FF)$ (for $p \geq 3$) and $(S,\widetilde{\FF})$ (for $p \geq 5$) are defined in Section~\ref{sec:plcg}. The saturation of $(S,\FF)$ and $(S,\widetilde{\FF})$ is proved in Theorem~\ref{FFpissaturated} and the exoticness results are proved in Theorems \ref{exoticaspcompact} and \ref{exoticasLiegroup}. The simplicity and normality conditions, together with the property of the fundamental groups, are proved in Proposition~\ref{prop_FFp_simple}. 
\end{proof}

Proving this result requires some new techniques related to $p$-local compact groups. Mainly, we develop the following tools.
\begin{itemize}
\item A saturation criterion for $p$-local compact groups (Proposition~\ref{1.1LO}) which generalizes \cite[Proposition~1.1]{LO} to the infinite case.
\item A classification of saturated fusion subsystems of index prime to $p$ in a given $p$-local compact group: Appendix~\ref{appendixA} generalizes the corresponding classification in the case of $p$-local finite groups (\cite{BCGLO2}) to the case of $p$-local compact groups.
\item An analysis of the homotopy type and connectedness of classifying spaces of certain \emph{centralizer $p$-local compact groups} associated to $p$-compact groups (Proposition~\ref{prop_p_compact}).
\end{itemize}

\begin{remark}\label{generated} 
	Let $S$ be a finite $p$-group or a discrete $p$-toral group. We say that $\FF$ is the fusion system over $S$ generated by $\Aut_\FF(Q)$ for $Q$ in a list of subgroups of $S$ if any other morphism in $\FF$ is the composition of restrictions of the given automorphism groups.
\end{remark}

\noindent\emph{Acknowledgements:} The authors thank the interest of C.~Broto, J.~M\o{}ller, B.~Oliver and A.~Viruel and the conversations with all of them while working in the results of this paper. Each author thanks respectively the Universitat Aut\`{o}noma de Barcelona, the Centre for Symmetry and Deformation in Copenhagen and the Kyoto University for their hospitality during the corresponding stays.


\section{On \texorpdfstring{$p$}{p}-local compact groups}\label{review_pcompact}

In this section we recall and generalize some concepts about $p$-local compact groups that we will use throughout this paper. For a more exhaustive treatment of this topic, the reader is referred to \cite{BLO2}, \cite{BLO3} and \cite{AKO}. Let $p$ be a prime to remain fixed for the rest of this section.

\subsection{Normal fusion subsystems}
We start by reviewing the concept of a normal subsystem of a fusion system. Let $S$ be a discrete $p$-toral group, and let $\FF$ be a saturated fusion system over $S$. 
Let also $P$ be a subgroup of $S$.
Recall that: 
\begin{itemize}
\item $P$ is \emph{strongly closed in $\FF$} if we have $\varphi(g)\in P$ for all $g \in P$ and $\varphi\in\Hom_\FF(\gen{g},S)$.
\item $P$ is \emph{normal in $\FF$} if $P$ is normal in $S$ and each morphism $\varphi \in \Hom_\FF(Q,R)$ in $\FF$ extends to a morphism $\overline{\varphi}\in\Hom_\FF(PQ,PR)$ such that $\overline{\varphi}(P)=P$. The maximal normal $p$-subgroup of $\FF$ is denoted by $O_p(\FF)$.
\end{itemize}

The following definition was introduced by M. Aschbacher \cite[Section 6]{Aschbacher} for finite fusion systems, although it applies to fusion systems over discrete $p$-toral groups without modification.

\begin{definition}\label{definormal}
Let $S$ be a discrete $p$-toral group, let $\FF$ be a saturated fusion system over $S$, and let $(S', \FF') \subseteq (S, \FF)$ be a subsystem. Then, $\FF'$ \emph{is normal in} $\FF$ if the following conditions are satisfied.
\begin{enumerate}[label=\textup{(N\arabic*)}]
\item $S'$ is strongly closed in $\FF$.
\item For each $P \leq Q \leq S'$ and each $\gamma \in \Hom_{\FF}(Q, S)$, the map that sends each $f \in \Hom_{\FF'}(P,Q)$ to $\gamma \circ f \circ 
\gamma^{-1}$ defines a bijection between the sets $\Hom_{\FF'}(P,Q)$ and $\Hom_{\FF'}(\gamma(P), \gamma(Q))$.
\item $\FF'$ is a saturated fusion system over $S'$.
\item Each $f \in \Aut_{\FF'}(S')$ extends to some $\widetilde{f} \in \Aut_{\FF}(S' C_S(S'))$ such that
\[
[\widetilde{f}, C_S(S')] = \{\widetilde{f}(g)\cdot g^{-1} \, | \, g \in C_S(S')\} \leq Z(S').
\]
\end{enumerate}
\end{definition}

\begin{lemma}\label{normal_subgroup_normal_subsystem}
Let $G$ be a compact Lie group and $S \in \Syl_p(G)$ be a Sylow $p$-subgroup. Let $H \trianglelefteq G$ be a normal closed subgroup and write $R = S \cap H$. Then $R \in \Syl_p(H)$ and the saturated fusion system $\FF_R(H)$ is normal in $\FF_S(G)$.
\end{lemma}

\begin{proof}
We have to show that $R \in \Syl_p(H)$ and that $(R, \FF_R(H)) \subseteq (S, \FF_S(G))$ satisfies the properties of Definition~\ref{definormal}.

To prove that $R \in \Syl_p(H)$, choose $P \in \Syl_p(H)$ such that $R \leq P$, which can be done by \cite[Proposition 9.3(b)]{BLO3}. By the same result from \cite{BLO3}, $P \leq S^g$ for some $g \in G$. Since $H \trianglelefteq G$ we get $P = P \cap H \leq S^g \cap H = (S \cap H)^g = R^g$. Therefore, $R \leq P \leq R^g$. By the discussion after Definition 1.1 in \cite{BLO3} we get $|R| \leq |P| \leq |R^g| = |R|$, so $R = P$.

In conclusion, we have that $R \in \Syl_p(H)$. It remains to prove that the fusion system $(R, \FF_R(H)) \subseteq (S, \FF_S(G))$ satisfies the properties of
Definition~\ref{definormal}.

\textbf{Condition (N1):} For every morphism $\varphi \in \FF_S(G)$ we have $\varphi = c_g$ for some $g \in G$. Since $H$ is normal in $S$ we obtain $\varphi(a) = c_g(a) \in R$ for all $a \in R$, so $R$ is strongly closed in $\FF_S(G)$.

\textbf{Condition (N2):} Fix $\gamma \in \Hom_{\FF_S(G)}(Q, S)$, again we have $\gamma = c_g$ for some $g \in G$. It is easy to see that we can define the map
\[\begin{array}{cccc}
\gamma^*\colon & \Hom_{\FF_R(H)}(P,Q) & \longrightarrow & \Hom_{\FF_R(H)}(\gamma(P), \gamma(Q))\\
 & f & \mapsto & \gamma \circ f \circ \gamma^{-1}
\end{array}\]
And, by considering $\beta = c_{g^{-1}}$, the map $\beta^*$ is an inverse of $\gamma^*$, so $\gamma^*$ is bijective.


\textbf{Condition (N3):} Since $H$ is a closed subgroup of $G$, it is itself a compact Lie group. Moreover, $R \in \Syl_p(H)$, so $\FF_R(H)$ is a saturated fusion system over $R$ by \cite[Lemma 9.5]{BLO3}.

\textbf{Condition (N4):} Consider the $p$-toral group $\overline{S}$, which is the topological closure of $S$. An elementary computation shows that $[N_H(R), C_{\overline{S}}(R)] \leq C_H(R)$. Moreover, 
we can see that $C_H(R)C_{\overline{S}}(R)$ is a normal subgroup of $N_H(R)C_{\overline{S}}(R)$. Then, since $R$ is strongly closed in $\FF_S(G)$, it is also fully centralized in $\FF_S(G)$, 
so we have $C_S(R) \in \Syl_p(C_G(R))$ by \cite[Lemma 9.5]{BLO3}. Now we have $C_S(R) \in \Syl_p(C_H(R)C_{\overline{S}}(R))$ and we can apply the Frattini argument to obtain
\[
N_H(R)C_{\overline{S}}(R) = C_H(R)C_{\overline{S}}(R)N_{N_H(R)C_{\overline{S}}(R)}(C_S(R)).
\]
Finally, let $f \in \Aut_{\FF_R(H)}(R)$, then $f = c_g$ with $g \in N_H(R)$. By the previous decomposition, we can write $g = xy$ for some $x \in 
C_H(R)C_{\overline{S}}(R)$ and $y \in N_{N_H(R)C_{\overline{S}}(R)}(C_S(R))$. So $y=hz$, with $h\in N_H(R)$ and $z\in C_{\overline{S}}(R)$. Under this situation, $h\in N_G(C_S(R))$: if $s\in C_S(R)$ and $r\in R$, we have to check that $(hsh^{-1})r(hsh^{-1})^{-1}=r$, and it follows because $hs(h^{-1}rh)s^{-1}h^{-1}=h(h^{-1}rh)h^{-1}$ (as $h^{-1}rh\in R$ and $s\in C_S(R)$). Then, we can consider $\widetilde{f} = c_h$, which defines an element $\widetilde{f} \in 
\Aut_{\FF_S(G)}(R C_S(R))$ and it is an extension of $f$. Moreover, if $g \in C_S(R)$, $\widetilde{f}(g) \cdot g^{-1} \in H \cap C_S(R) = C_R(R) = Z(R)$.
\end{proof}

Next we recall the definition of simplicity for saturated fusion systems and $p$-local compact groups \cite[Definition~3.1 and Remark~2.3]{Gonza2}.

\begin{definition}\label{defisimple}
Let $\FF$ be a saturated fusion system over a discrete $p$-toral group $S$. Then, $\FF$ is \emph{simple} if it satisfies one of the following conditions.
\begin{enumerate}[label=\textup{(\roman*)}]
\item $\rk(\FF) = 0$ and $\FF$ has no proper normal subsystems.

\item $\rk(\FF) \geq 1$ and every proper normal subsystem of $\FF$ is finite.

\end{enumerate}
Similarly, a $p$-local compact group $\g = (S, \FF, \LL)$ is \emph{simple} if $\FF$ is simple.

\end{definition}

Note that if $\g = (S, \FF, \LL)$ is a simple $p$-local compact group of positive rank $r \geq 1$, then the fusion system $\FF$ may still contain finite normal subsystems. The reader may compare this situation with the definitions of finite simple group and simple compact Lie group, where a similar phenomenon occurs.

For a (possibly infinite) group $G$, let $O^p(G)$ be the intersection of all the normal subgroups of $G$ that have finite $p$-power index. Similarly, let $O^{p'}(G)$ be the intersection of all the normal subgroups of $G$ that have finite index prime to $p$. Following the usual notation for groups, let also $O_p(G)$ be the maximal normal $p$-subgroup of $G$.

\begin{definition}\label{defifoca}
Let $S$ be a discrete $p$-toral group, and let $\FF$ be a saturated fusion system over $S$. The \emph{hyperfocal subgroup of $\FF$} is the subgroup
\[
O_{\FF}^p(S) = \gen{T, \{g \cdot \varphi(g)^{-1} \, |\, g \in Q \leq S, \, \varphi \in O^p(\Aut_{\FF}(Q))\}} \leq S.
\]
Given a saturated subsystem $(S', \FF') \subseteq (S, \FF)$, we say that:
\begin{itemize}
\item $\FF'$ has \emph{$p$-power index in $\FF$} if $S' \geq O_{\FF}^p(S)$, and $\Aut_{\FF'}(P) \geq O^p(\Aut_{\FF}(P))$ for all $P \leq S'$.

\item $\FF'$ has \emph{index prime to $p$} if $S' = S$, and $\Aut_{\FF'}(P) \geq O^{p'}(\Aut_{\FF}(P))$ for all $P \leq S'$.

\end{itemize}

\end{definition}

In some parts of this paper we deal with subsystems of $\FF$ of \emph{$p$-power index} and of \emph{index prime to $p$}. Such subsystems can be detected via the computation of certain subgroups. The detection techniques in the finite case were developed in \cite{BCGLO2}. In the compact case, the tools for the detection of subsystems of $p$-power index were developed in \cite[Appendix B]{Gonza2}. Regarding the detection of subsystems of index prime to $p$, the necessary techniques are developed in Appendix \ref{appendixA}.

The following result states the classification of subsystems of $p$-power index and of index prime to $p$ for a given infinite fusion system. The first part corresponds to \cite[Theorem B.12]{Gonza2}, together with \cite[Corollary B.13]{Gonza2}, and the second part corresponds to Theorem~\ref{app4}, together with Corollary~\ref{minnor}.

\begin{theorem}
Let $(S, \FF, \LL)$ be a $p$-local compact group, and set
\[
\Gamma_{p'}(\FF) \definicio \pi_1(|\LL|)/O^{p'}(\pi_1(|\LL|))
\]
(see Equation~\eqref{Gamma} in the Appendix~\ref{appendixA} for more details). Then, the following holds.
\begin{enumerate}
\item For each $R \leq S$ such that $O_{\FF}^p(S) \leq R$, there exists a unique subsystem $(R,\FF_R) \subseteq (S,\FF)$ of $p$-power index. In particular, 
$\FF$ contains a minimal subsystem of $p$-power index, denoted by $O^p(\FF)$, which is normal in $\FF$.
\item For each $H \leq \Gamma_{p'}(\FF)$ there exists a unique subsystem $(S, \FF_H) \subseteq (S, \FF)$ of index prime to $p$. In particular, $\FF$ contains a minimal subsystem of index prime to $p$, denoted by $O^{p'}(\FF)$, which is normal in $\FF$.
\end{enumerate}
\end{theorem}

\begin{definition}
The saturated fusion system $(S, \FF)$ is \emph{reduced} if $O_p(\FF)=1$ and $O^p(\FF)=O^{p'}(\FF)=\FF$.
\end{definition}

\begin{lemma}\label{reducedsimple}
Let $(S, \FF)$ be a saturated fusion system, with $S$ a finite $p$-group. If $\FF$ is reduced and $S$ contains no proper strongly $\FF$-closed subgroups, then $\FF$ is simple.
\end{lemma}

\begin{proof}
Suppose $\FF$ is not simple, and let $(R, \mathcal{E}) \lhd (S,\FF)$ be a (nontrivial) normal subsystem. Then, $R$ is a strongly $\FF$-closed subgroup of $S$. As $\mathcal{E}$ is nontrivial and $S$ contains no proper strongly $\FF$-closed subgroups, it follows that $R = S$. By \cite[Lemma 5.72]{Craven}, $\Aut_{\mathcal{E}}(Q)$ has index prime to $p$ in $\Aut_{\FF}(Q)$ for all $Q \leq S$. Hence, $\mathcal{E}$ must be of index prime to $p$ in $\FF$, contradicting the assumption that $O^{p'}(\FF) = \FF$.
\end{proof}

\subsection{A saturation criterion}

In this subsection we present a generalization of \cite[Proposition 1.1]{LO} for infinite fusion systems (see \cite[Proposition 4.4]{BM} for a topological analogue of the result in \cite{LO}). Before we prove the main result of this subsection, Proposition~\ref{1.1LO}, we need a technical result.


\begin{lemma}\label{lemma}
	Let $S$ be a discrete $p$-toral group and $P \leq S$ any nontrivial subgroup. Then, there exists an element $x \in Z(P)$ of
order $p$ which is fixed by all $\varphi \in \Aut_S(P)$.
\end{lemma}

\begin{proof}
As $\Aut_S(P) = \Aut_{N_S(P)}(P)$, we may assume that $P \lhd S$. Suppose first that $P$ is finite. In this case, $\Aut_S(P) = S/C_S(P)$ is a finite $p$-group, and it is a well-known fact that $1 \neq P^{S/C_S(P)} \leq Z(P)$. Suppose now that $P$ is not finite. In this case, consider the quotient $S/C_S(P)T$, where $T \leq S$ is the maximal torus. Let also $T'$ be the maximal torus of $P$, and let $\Omega_1(T') \leq T'$ be the subgroup generated by all elements of order $p$. Since $\Omega_1(T')$ is a characteristic subgroup of $P$ and since $T' \leq T$ is abelian, it follows that $S/C_S(P)T$ is a finite $p$-group which acts on $\Omega_1(T')$. By the finite case studied above, it follows that $1 \neq \Omega_1(T')^{S/C_S(P)T} \leq Z(P)$. This finished the proof.
\end{proof}

\begin{proposition}\label{1.1LO}
	Let $(S, \FF)$ be a fusion system over a discrete $p$-toral group. Then, $\FF$ is saturated if and only if the following holds.
	\begin{enumerate}
	\item $(S, \FF)$ satisfies axiom \textup{(III)} in \cite[Definition 2.2]{BLO3}: let $P_1\leq P_2\leq P_3\leq \cdots $ be an increasing sequence of subgroups of $S$, with $P_\infty=\bigcup_{n=1}^\infty P_n$, and if $\varphi\in\Hom(P_\infty,S)$ is any homomorphims such that $\varphi|_{P_n}\in\Hom_{\FF}(P_n,S)$ for all $n$, then $\varphi\in\Hom_{\FF}(P_\infty,S)$.
	\item There exists a set $\setp$ of elements of order $p$ in $S$ such that the following conditions are satisfied:
	\begin{enumerate}[label=\textup{(\roman*)}]
		\item each $x \in S$ of order $p$ is $\FF$-conjugate to some $y \in \setp$;
		\item if $x,y$ are $\FF$-conjugate and $y \in \setp$, then there is some morphism \[\rho \in \Hom_{\FF}(C_S(x), C_S(y))\] such that 
$\rho(x) = y$ and
		\item for each $x \in \setp$, $C_{\FF}(x)$ is a saturated fusion system over $C_S(x)$.
	\end{enumerate}
	\end{enumerate}
\end{proposition}
\begin{proof}
Suppose first that $(S,\FF)$ is saturated. Then, condition (a) above is obviously satisfied. To check condition (b), let $\setp$ be the set of all elements $x \in S$ of order $p$ such that $\gen{x}$ is fully centralized in $\FF$. Then, conditions (i) and (ii) follow immediately from the saturation axioms on $\FF$ and condition (iii) is \cite[Theorem 2.3]{BLO6}.

Suppose now that $(S,\FF)$ satisfies conditions (a) and (b) in the statement, for a certain set $\setp$. We have to show that $(S, \FF)$ satisfies axioms (I) and (II) in \cite[Definition 2.2]{BLO3}. For the reader's convenience, we recall the statement of axiom before proving it. Before checking that axioms (I) and (II) are satisfied, we show that $\Aut_{\FF}(P)$ has Sylow $p$-subgroups for all $P \leq S$.

More precisely, given $P \leq S$ we show that $\Aut_{\FF}(P)$ has a normal subgroup of finite index which is isomorphic to a discrete $p$-torus. Let $x \in Z(P)$ be an element of order $p$. As $\Aut_{\FF}(P) \cong \Aut_{\FF}(P')$ if $P'$ is $\FF$-conjugate to $P$, by conditions (i) and (ii) in the statement we may assume that $x \in \setp$. Consider the subgroup $\Omega_1 = \Omega_1(Z(P))$, generated by all elements of order $p$ in $Z(P)$. This is a characteristic subgroup of $P$, and hence $\Aut_{\FF}(P)$ acts on $\Omega_1$. Moreover, $\Omega_1$ is a finite subgroup of $Z(P)$, and this implies that $C_{\Aut_{\FF}(P)}(\Omega_1)$ is a normal subgroup of $\Aut_{\FF}(P)$ of finite index.

Note that $C_{\Aut_{\FF}(P)}(\Omega_1) \leq \Aut_{C_{\FF}(P)}(P)$ just by definition of $C_{\FF}(x)$. As the latter is a saturated fusion system by condition (iii) in the statement, we know that $\Out_{C_{\FF}(x)}(P)$ is a finite group, and this implies that $\Aut_{C_{\FF}(P)}(P)$ contains a normal subgroup of finite index which is isomorphic to a discrete $p$-torus. Let $H \leq \Aut_{C_{\FF}(P)}(P)$ denote such subgroup. Since $\Aut_{C_{\FF}(P)}(P)$ has finite index in $\Aut_{\FF}(P)$, it follows that $H$ also has finite index in $\Aut_{C_{\FF}(P)}(P)$. Furthermore, as $H$ is both infinitely $p$-divisible and a discrete $p$-toral subgroup of $\Aut_{C_{\FF}(P)}(P)$, it follows that $H$ is normal in $\Aut_{\FF}(P)$.

Notice that the subgroup $H$ above is a subgroup of $\Aut_{C_S(x)}(P) \leq \Aut_S(P)$, and hence it follows that $\Out_{\FF}(P)$ is a finite group. In addition, \cite[Lemma 8.1]{BLO3} applies now to show that $\Aut_{\FF}(P)$ has Sylow $p$-subgroups. We are ready to prove axioms (I) and (II). We first prove axiom (II), as we need it when proving axiom (I).

\textbf{Axiom (II):} If $P \leq S$ and $\varphi \in \Hom_\FF(P, S)$ are such that $\varphi(P)$ is fully $\FF$-centralized, and if we set
\[
N_\varphi = \{ g \in N_S(P) \mid \varphi \circ c_g \circ \varphi^{-1} \in \Aut_S(P') \},
\]
then there is $\widetilde{\varphi} \in \Hom_{\FF}(N_{\varphi}, S)$ such that $\widetilde{\varphi}|_P = \varphi$.

Choose $x' \in Z(P')$ of order $p$ and which is fixed under the action of $\Aut_S(P')$, which exists by Lemma~\ref{lemma}. Write $x = 
\varphi^{-1}(x') \in
Z(P)$ and note that, for all $g \in N_\varphi$, the morphism $\varphi \circ c_g \circ \varphi^{-1} \in \Aut_S(P')$ fixes $x'$, thus $c_g(x) = x$. 
Hence,
\begin{itemize}
\item[(A)] $x \in Z(N_\varphi)$, which implies $N_\varphi \leq C_S(x)$. Also, we have $N_S(P') \leq C_S(x')$.
\end{itemize}
Let $y \in \setp$ be $\FF$-conjugate to $x$ and $x'$, whose existence is guaranteed by property (i) of the set $\setp$. Also, by property (ii) of $\setp$, there exist $\rho \in \Hom_\FF(C_S(x), C_S(y))$ and $\rho' \in \Hom_\FF(C_S(x'), C_S(y))$ such that $\rho(x) = y = \rho'(x')$. Set also $Q = \rho(P)$ 
and $Q' =
\rho'(P')$. Since $P$ is fully $\FF$-centralized and $C_S(P) \leq C_S(x)$, it follows that
\begin{itemize}
\item[(B)] $\rho'(C_{C_S(x')}(P')) = \rho'(C_S(P')) = C_S(Q') = C_{C_S(y)}(Q')$.
\end{itemize}
Set $\omega = \rho' \circ \varphi \circ \rho^{-1} \in \Iso_\FF(Q, Q')$. By construction, $\omega(y) = y$, and thus $\omega \in \Iso_{C_\FF(y)}(Q, Q')$. 
Since $P'$ is
fully centralized in $\FF$, property (B) implies that $Q'$ is fully centralized in $C_\FF(y)$. Then, we can apply axiom (II) of saturated fusion systems on 
$\omega$ as a
morphism in $C_\FF(y)$, which is a saturated fusion system by property (iii) of $\setp$. We obtain that $\omega$ extends to some 
$\widetilde{\omega} 
\in
\Hom_{C_\FF(y)}(N_\omega, C_S(y))$, where
\[
N_\omega = \{ g \in N_{C_S(y)}(Q) \mid \omega \circ c_g \circ \omega^{-1} \in \Aut_{C_S(y)}(Q') \}.
\]
Note that, for all $g \in N_\varphi \leq C_S(x)$, we have, by property (A):
\begin{align*}
c_{\omega(\rho(g))} &= \omega \circ c_{\rho(g)} \circ \omega^{-1} = (\omega \circ \rho) \circ c_g \circ (\omega \circ \rho)^{-1} = \\
&= (\rho' \circ \varphi) \circ c_g \circ (\rho' \circ \varphi)^{-1} = c_{\rho'(h)} \in \Aut_{C_S(y)}(Q')
\end{align*}
for some $h \in N_S(P')$ such that $\varphi \circ c_g \circ \varphi^{-1} = c_h$. In particular, $\rho(N_\varphi) \leq N_\omega$. Then, we obtain $\omega(\rho(g)) = \rho'(h) l'$, for some $l' = C_S(Q')$. By property (B) we know that $C_S(Q') = \rho'(C_S(P'))$, so in fact $\omega(\rho(g)) = \rho'(h)\rho'(l)$. Since $N_S(P') \leq C_S(x')$, we obtain $\widetilde{\omega}(\rho(N_\varphi)) \leq \rho'(N_{C_S(x)}(P'))$. We can then define
\[
\widetilde{\varphi} = (\rho')^{-1} \circ (\widetilde{\omega} \circ \rho)|_{N_\varphi} \in \Hom_\FF(N_\varphi, S)
\]
which clearly satisfies axiom (II) above.

\textbf{Axiom (I):} For all $P \leq S$ which is fully normalized in $\FF$, $P$ is fully centralized in $\FF$, $\Out_{\FF}(P)$ is finite, and $\Out_S(P) \in \Syl_p(\Out_{\FF}(P))$.

Let $P \leq S$ be fully normalized in $\FF$, with $P \neq 1$. As $\Aut_{\FF}(P)$ has Sylow $p$-subgroups, we have
\[
\Out_S(P) \in \Syl_p(\Out_{\FF}(P)) \Longleftrightarrow \Aut_S(P) \in \Syl_p(\Aut_{\FF}(P)).
\]
Before we proceed with the rest of the proof of axiom (I), we need to define two sets and to prove two auxiliary results. Consider the sets $U$ and $U_0$ defined as
\begin{flalign*}
U = \lbrace (P,x) \mid & P \leq S \text{ is finite, } \exists \Gamma \in \Syl_p(\Aut_{\FF}(P)) \text{ such that } \Aut_S(P) \leq \Gamma \text{ and }\\
 & x \in Z(P)^{\Gamma} \text{ has order } p \rbrace \text{ and }\\
U_0 = \lbrace &(P, x) \in U \mid x \in \setp \rbrace .
\end{flalign*}
Note that for each nontrivial finite subgroup $P \leq S$, there is some $x \in P$ such that $(P, x) \in U$, since every action of a finite $p$-group 
on $Z(P)$ has nontrivial fixed set. Then, we have the following:
\begin{itemize}
\item[(C)] If $(P, x) \in U_0$ and $P$ is fully centralized in $C_\FF(x)$, then $P$ is fully centralized in $\FF$.
\end{itemize}
Assume otherwise and let $P' \in P^\FF$ be fully centralized in $\FF$ and $\varphi \in \Iso_\FF(P, P')$. Write also $x' = \varphi(x) \in Z(P')$. By 
property
(ii) of the set $\setp$, there is $\rho \in \Hom_\FF(C_S(x'), C_S(x))$ such that $\rho(x') = x$, since we are assuming $x \in \setp$. Note that $P' 
\leq C_S(x')$ and set then $P'' = \rho(P')$. In particular, $\rho \circ \varphi \in \Iso_{C_\FF(x)}(P, P'')$ and therefore $P''$ is 
$C_\FF(x)$-conjugate to $P$. Also, since $\gen{x'} \leq P'$, we have $C_S(P') \leq C_S(x')$ and then $\rho$ sends $C_S(P')$ injectively into 
$C_S(P'')$. Hence,
\[
|C_S(P)| < |C_S(P')| \leq |C_S(P'')|.
\]
However, the equalities $C_S(P) = C_{C_S(x)}(P)$ and $C_S(P'') = C_{C_S(x)}(P'')$ contradict the assumption that $P$ is fully centralized in
$C_\FF(x)$, and this proves property (C).

Note that, by definition, $N_S(P) \leq C_S(x)$ for all $(P, x) \in U$, and hence
\[
\Aut_{C_S(x)}(P) = \Aut_S(P).
\]
Also, if $(P, x) \in U$ and $\Gamma \in \Syl_p(\Aut_\FF(P))$ is as in the definition of $U$, then $\Gamma \leq \Aut_{C_\FF(x)}(P)$. In particular, we 
have
\begin{itemize}
\item[(D)] For all $(P, x) \in U$,
\[
\Aut_S(P) \in \Syl_p(\Aut_\FF(P)) \Longleftrightarrow \Aut_{C_S(x)}(P) \in \Syl_p(\Aut_{C_\FF(x)}(P)).
\]
\end{itemize}

We are ready to check that $\FF$ satisfies axiom (I) of saturated fusion systems. Fix $P \leq S$, $P \neq 1$, a finite subgroup fully normalized in $\FF$. By definition, $|N_S(P)| \geq |N_S(P')|$ for all $P' \in P^\FF$. Choose $x \in Z(P)$ such 
that $(P,
x) \in U$ and let $\Gamma \in \Syl_p(\Aut_\FF(P))$ be such that $\Aut_S(P) \leq \Gamma$ and such that $x \in Z(P)^\Gamma$. Then, by properties (i) 
and (ii) of the set $\setp$, there is some $y \in \setp$ and $\rho \in \Hom_\FF(C_S(x), C_S(y))$ such that $\rho(x) = y$. Set $P' = \rho(P)$ and 
$\Gamma' = \rho \circ \Gamma \circ \rho^{-1} \in \Syl_p(\Aut_\FF(P'))$.

Note that, since $P$ is assumed to be fully normalized in $\FF$, it follows that $\rho(N_S(P)) = N_S(P')$. Therefore, $\Aut_S(P') \leq \Gamma'$ and $y 
\in
Z(P')^{\Gamma'}$, hence $(P', y) \in U_0$. Since $N_S(P') \leq C_S(y)$, the maximality of $|N_S(P')| = |N_{C_S(y)}(P')|$ implies that $P'$ is fully 
normalized
in $C_\FF(y)$.

Now, by property (iii) of $\setp$, the fusion system $C_\FF(y)$ is saturated. Then, since $P'$ is fully normalized in $C_\FF(y)$, we have that $P'$ 
is 
fully centralized in $C_\FF(y)$ and $\Aut_{C_S(y)}(P') \in \Syl_p(\Aut_{C_\FF(y)}(P'))$. Therefore, by properties (C) and (D), $P'$ is fully centralized in 
$\FF$ and $\Aut_S(P') \in \Syl_p(\Aut_\FF(P'))$.

Recall that $P$ is fully $\FF$-normalized. Since $P'$ is fully centralized in $\FF$, and since we have already proved that axiom (II) holds on $\FF$, we may apply \cite[Lemma~2.2]{BLO6} to deduce that $P$ is also fully centralized in $\FF$ and $\Aut_S(P) \in \Syl_p(\Aut_\FF(P))$. Finally, it is shown in \cite{BLO3} that if axiom (I) holds for all finite fully normalized subgroups, then axiom (I) holds for all fully normalized subgroups.
\end{proof}


\section{Some families of exotic \texorpdfstring{$p$}{p}-local finite groups }\label{sec:finite}

In this section we present some examples of exotic $p$-local finite groups, for $p \geq 3$. That is, we describe some $p$-local finite groups which are not realized by any finite group. The examples in this section are organized as follows. For $p = 3$, we present a family $\{(S_k, \FF_k, \LL_k)\}_{k \geq 2}$. This family was first studied in \cite{DRV}, where the exoticness was also proved. Inspired by the family for $p = 3$, we construct, for $p \geq 5$, two families $\{(S_k, \FF_k, \LL_k)\}_{k \geq 2}$ and $\{(S_k, \widetilde{\FF}_k, \widetilde{\LL}_k)\}_{k \geq 2}$. As it turns out, the latter family was already studied in \cite{BLO2}, where the authors also prove exoticness.

\subsection{A family of exotic \texorpdfstring{$3$}{3}-local finite groups}\label{sec:finite3}

For this subsection we focus on the prime $p = 3$. We are interested in the $3$-local finite groups corresponding to the saturated fusion systems over $S_k$ denoted in \cite[Table 6]{DRV} as $\FF(3^{2k+1},3)$, with $k \geq 2$. These saturated fusion systems are over $3$-groups which can be expressed as an split extension (see Notation~\ref{defiTk} below):
\begin{equation}\label{eq:S_3}
1 \longrightarrow T_k \longrightarrow S_k \longrightarrow C_3 \longrightarrow 1
\end{equation}
where $T_k \cong (C_{3^k})^2$, with a fixed (non-trivial) action of $C_3$ on $T_k$. These groups are noted as $B(3,2k+1;0,0,0)$ in \cite[Appendix A]{DRV}. As we present in later sections a generalization of these $3$-groups for all odd primes, we do not give too many details and, instead, we refer the reader to \cite{DRV} for further details. Thus, consider the action of $\GL_2(\F_3)$ on $T_{k}$ as explained in \cite[Lemma A.17]{DRV}, $\zeta$ a generator of the center of $S_{k}$, $s$ an element (of order $3$) not in $T_{k}$ and $V\definicio \langle \zeta,s \rangle$ an elementary abelian $3$-group of rank $2$.


\begin{lemma}\label{elements-not-in-maximal-torus-conjugated-to-s}
Let $s_1^i s_2^j s \in S_{k}$ be an element not in the maximal torus. Then, $s_1^i s_2^j s$ is conjugated to $s$ if and only if $i \equiv 0 \mod{3}$.
\end{lemma}

\begin{proof}
Note that if we conjugate $s$ by any element $s_1^\alpha s_2^\beta s \in S_{k}$ we obtain $s_1^{3\beta} s_2^{-\alpha + 3\beta} s$, and $3\beta 
\equiv 0 \mod{3}$. Conversely, if $i = 3l$, conjugating $s$ by $s_1^{i-j} s_2^l s$ we obtain $s_1^i s_2^j s$.
\end{proof}

\begin{theorem}[\cite{DRV}]
Let $S_k$ a finite $3$-group as in Equation~\eqref{eq:S_3} and consider $\eta$, $\omega$ the outer autmorphisms of $S_k$ as in \cite[Notation~5.9]{DRV}.
The following automorphisms groups generate simple saturated fusion systems $(S_k, \FF_k) = (S_{k},\FF(3^{2k+1},3))$ in the sense of Remark~\ref{generated}:
\begin{itemize}
\item $\Aut_{\FF_k}(S_k)=\langle \eta, \omega, \Inn(S_k)\rangle$, getting $\Out_{\FF_k}(S_k)\cong C_2 \times C_2$,
\item $\Aut_{\FF_k}(T_{k})\cong \GL_2(\F_3)$ and
\item $\Aut_{\FF_k}(V))\cong \GL_2(\F_3)$.
\end{itemize}
Moreover, $S_k$, $T_k$ and $V$ are representatives of the only $\FF_k$-conjugacy classes of centric radical subgroups of $S_k$.
\end{theorem}
\begin{proof}
The saturation and exoticness properties are proven in \cite[Theorem 5.10]{DRV}. For the simplicity property, note that $\FF(3^{2k+1},3)$ has no proper 
nontrivial strongly closed subgroups. Indeed, let $P \trianglelefteq S_{k}$ a nontrivial strongly closed subgroup. By \cite[Theorem 8.1]{AB}, $P$ 
must intersect the center in a nontrivial subgroup. Since the center of $S_{k}$ has order 3, we must have $Z(S_{k}) \leq P$. Moreover, since 
$\zeta$ is $\FF(3^{2k+1},3)$-conjugated to $s$, we must have also $s \in P$, since $P$ is strongly closed. Then, by \cite[Lemma 2.2]{Blackburn}, $P$ must be 
of index at most $3$ in $S_{k}$.


In fact, by Lemma~\ref{elements-not-in-maximal-torus-conjugated-to-s}, $P$ must contain the subgroup generated by $s$ and all elements $s_1^{3l}s_2^j \in 
T_{k}$, which is an index 3 subgroup of $S_{k}$. Then, using that the automorphism group of $T_{k}$ is all $\GL_2(\F_3)$, we can conjugate, for 
example, the element $s_1^{-3}s_2^{-2}$ to $s_1s_2$, obtaining that $P$ contains also elements not conjugated to $s$. Since $P$ had index at most 3, 
we obtain that $P$ must be equal to $S_{k}$.

Hence, if there is a proper nontrivial normal subsystem of $\FF(3^{2k+1},3)$, it has to be over the same group $S_{k}$, by condition (N1) of Definition~\ref{definormal}. Then, by \cite[Lemma 5.72]{Craven}, we have that the normal subsystem has to be of index prime to 
$p$ in $\FF(3^{2k+1},3)$, but by the classification in \cite[Theorem 5.10]{DRV}, there is no subsystem of index prime to $p$ in $\FF(3^{2k+1},3)$.
\end{proof}

\subsection{Some \texorpdfstring{$p$}{p}-groups of maximal class}
The Sylow $3$-subgroups listed in the previous subsection are particular cases of maximal nilpotency class finite $p$-groups. These groups were classified by N.~Blackburn and fit in a family defined for every prime $p\geq 3$ in \cite[Page 88]{Blackburn} . We recall here the presentation given there.

Let $p$ be an odd prime, to remain fixed for the rest of this section unless otherwise specified. For $k \geq 2$, define $S_{k}$ as the group of order $p^{(p-1)k+1}$ with parameters $\alpha=\beta=\gamma=\delta=0$
in \cite{Blackburn}. This means that $S_{k}$ can be given by the following presentation: $\{s,s_1, s_2, \dots, s_{(p-1)k}\}$ is a generating set, with relations
\begin{flalign}
& [s,s_{i-1}] = s_i   \mbox{ for $i=2, \dots, (p-1)k$}\label{commutadors}\\
& [s_1,s_i] =  1   \mbox{ for $i=2, \dots, (p-1)k$}\label{s1commuta}\\
& s^p  =  1  \label{ordres}\\
& s_i^{\binom{p}{1}} s_{i+1}^{\binom{p}{2}} \cdots s_{i+p-1}^{\binom{p}{p}} =  1  \mbox{ for $i=1, \dots,
(p-1)k$}\label{relacionssi}
\end{flalign}
In the last equation, we are assuming $s_j=1$ for $j>(p-1)k$.

Consider 
$\gamma(S_{k})\definicio \langle s_i \rangle_{1\leq i \leq (p-1)k}$
 and the lower central series \[\gamma_2(S_{k})=[S_{k},S_{k}], \, 
\gamma_i(S_{k})=[\gamma_{i-1}(S_{k}),S_{k}].\]
The following proposition gives two properties of $S_{k}$ which will allow us to see it as an extension of a finite torus by an element of order $p$:
\begin{proposition}\label{propSpk} The following holds. 
\begin{enumerate}
\item The subgroup $\gamma(S_{k})$ is isomorphic to $(C_{p^k})^{p-1}$
with generators $s_1, \dots, s_{p-1}$.
\item There are $p$ conjugacy classes of subgroups of order $p$ not contained in $\gamma(S_{k})$. The subgroups $\langle s s_1^i\rangle$, for $i \in \{0, 1, \dots, p-1\}$, form a set of representatives of each conjugacy class.
\end{enumerate}
\end{proposition}
\begin{proof}
Let us see first that the center of $\gamma \definicio \gamma(S_{k})$, that we denote by $Z(\gamma)$, is all of $\gamma$.
From Equation~\eqref{s1commuta} we obtain that $s_1\in Z(\gamma)$. Conjugation by $s$ induces an automorphism of $\gamma$, so $c_s(Z(\gamma))=Z(\gamma)$.
From Equation~\eqref{commutadors}, $c_s(s_1)=s_1s_2 \in Z(\gamma)$, which implies that $s_2\in Z(\gamma)$. 
As $c_s(s_i)=s_is_{i+1}$ we can iterate this argument and $Z(\gamma)=\gamma$. So $\gamma$ is abelian.

Now we have to classify $\gamma$, an abelian group generated by $\{ s_i \}_{1\leq i \leq (p-1)k}$, subject to $(p-1)k$ relations given by Equation~\eqref{relacionssi}. Using additive notation, these elements are given by the $(p-1)k\times (p-1)k$ upper triangular matrix:
$$
R=\begin{pmatrix}
\binom{p}{1} & \binom{p}{2} & \cdots & \binom{p}{p} & 0 & \cdots & 0 \\
0 & \binom{p}{1} & \cdots & \binom{p}{p-1} & \binom{p}{p} & \cdots & 0 \\
\vdots & \vdots  & \ddots & \vdots & \vdots & \ddots & \vdots \\
0 & 0 & \cdots & \binom{p}{1} & \binom{p}{2} & \cdots & \binom{p}{p} \\
0 & 0 & \cdots & 0 & \binom{p}{1} & \cdots & \binom{p}{p-1} \\
\vdots & \vdots  & \ddots & \vdots & \vdots & \ddots & \vdots \\
0 & 0 & \cdots & 0 & 0 & \cdots & \binom{p}{1}
\end{pmatrix}
$$
where $\binom{p}{1}=p$, $p$ divides $\binom{p}{k}$ if  $k\not\in\{0,p\}$ and $\binom{p}{p}=1$. Then:
\begin{itemize}
\item The order of $\gamma$ is the determinant of $R$, which is equal to $p^{(p-1)k}$. 
\item If we consider the bottom-right $(p-1)\times(p-1)$ submatrix of $R$, with all coefficients in the diagonal equals $p$ and all coefficients divisible by $p$, we get that $\langle s_m \rangle_{(p-1)(k-1)+1\leq m\leq (p-1)k}$ generate a subgroup which is a quotient of $(C_p)^{p-1}$. 
\item Consider now the bottom-right $2(p-1)\times2(p-1)$ submatrix of $R$, with all the coefficients $1$ in the last non-zero overdiagonal. This implies that $s_m$, for ${(p-1)(k-1)+1\leq m\leq (p-1)k}$ is a combination of $\langle s_m^p \rangle_{(p-1)(k-2)+1\leq m\leq (p-1)(k-1)}$, so  $\langle s_m\rangle_{(p-1)(k-2)+1\leq m\leq (p-1)(k-1)}=\langle s_m\rangle_{(p-1)(k-2)+1\leq m\leq (p-1)k}$ generate a quotient of $(C_{p^2})^{p-1}$.
\item Iterating this process, we get that $\gamma$ is a quotient of $(C_{p^k})^{p-1}$.
\end{itemize}
This quotient relation, together with the computation of the order of $\gamma$, implies that $\gamma = \langle s_m \rangle_{1\leq m \leq (p-1)} \cong (C_{p^k})^{p-1}$.

For the second part of the statement, notice that using Equation~\eqref{commutadors} we get that the element $s$ is conjugated to
$s s_2^{i_2} s_3^{i_3} \cdots  s_{p-1}^{i_{p-1}}$ for all $0\leq i_j \leq p^{k}-1$. Use now Equation~\eqref{relacionssi} for $i=1$ to see that conjugating $s$ by powers of $s_{p-1}$ we obtain all possible elements of the type $s s_1^{p i_1}$ modulo $\gamma_2(S_{k})$. Moreover $s$ is not conjugated to $s s_1^i$ for $1\leq i \leq p-1$ and then $s s_1^i$ is not conjugated to $s s^j$ for $i$ and $j$ such that $p\nmid (j-i)$.  Moreover $(s s_1^i)^p=1$: \cite[Page 83, Equation (39)]{Blackburn} with $\beta=0$, getting that $(s s_1^i)^p=s^p(s_1^ps_2^{\binom{p}{2}} \cdots s_p)^i$, which is equal to $1$ by Equations \eqref{ordres} and \eqref{relacionssi}. The result follows from the fact that these are representatives of all possible conjugacy classes of elements of order $p$ which are not in $\gamma(S_{k})$ 
\end{proof}

This proposition tells us that $S_{k}$ fits in a split extension:
\begin{equation}\label{extensionS}
1 \rightarrow (C_{p^k})^{p-1} \rightarrow S_{k} \rightarrow C_p \rightarrow 1
\end{equation}
and the action of $\langle s \rangle \cong C_p$ over $(C_{p^k})^{p-1}$ is given by the matrix
(using additive notation, hence with coefficients in $\Z/p^k$, on the generators $\{s_1, \dots, s_{p-1}\}$):
\begin{equation}
\label{action_A}
A=
\begin{pmatrix}
1 & 0 & 0 & \cdots & 0 & -{\binom{p}{1}} \\
1 & 1 & 0 & \cdots & 0 & -{\binom{p}{2}} \\
0 & 1 & 1 & \cdots & 0 & -{\binom{p}{3}} \\
\vdots & \vdots & \vdots & \ddots & \vdots & \vdots \\
0 & 0 & 0 & \cdots & 1 & -{\binom{p}{p-2}} \\
0 & 0 & 0 & \cdots & 1 & 1-{\binom{p}{p-1}}
\end{pmatrix}.
\end{equation}
Consider now $(\Z/p^k)^p$ generated by $e_1, \dots , e_p$ and the action of $\Sigma_p$, the symmetric group on $p$ letters, by permutation of
the elements of the basis. This action leaves invariant the submodule $M_k$, isomorphic to $(\Z/p^k)^{p-1}$,
generated by the basis $\langle v_1, \dots, v_{p-1}\rangle$, where
$v_1 \definicio e_1-e_2$, $v_2\definicio e_2-e_3$, \ldots , $v_{p-1}\definicio e_{p-1}-e_p$, so we get 
an action of $\Sigma_p$ on $(\Z/p^k)^{p-1}$, and allows us to
construct a split extension:
\begin{equation}
\xymatrix{
	1 \ar[r] & (C_{p^k})^{p-1} \ar[r] & (C_{p^k})^{p-1}\rtimes \Sigma_p \ar[r]  & \Sigma_p \ar[r] & 1
}.
\label{actionsigma}
\end{equation}
In the basis $\{v_1, v_2, \dots, v_{p-1}\}$ the permutation 
$(1,2, \dots, p)$ corresponds to the matrix:
\begin{equation}
\label{action_B}
B=
\begin{pmatrix}
0 & 0 & 0 & \cdots & 0 & -1 \\
1 & 0 & 0 & \cdots & 0 & -1 \\
0 & 1 & 0 & \cdots & 0 & -1 \\
\vdots & \vdots & \vdots & \ddots & \vdots & \vdots \\
0 & 0 & 0 & \cdots & 0 & -1 \\
0 & 0 & 0 & \cdots & 1 & -1
\end{pmatrix}
\end{equation}
\begin{lemma}
The actions of $C_p$ on $(C_{p^k})^{p-1}$ induced by matrices $A$ and $B$ from Equations \eqref{action_A} and \eqref{action_B} produce isomorphic split extensions of the form
\[
\xymatrix{
1 \ar[r] & (C_{p^k})^{p-1} \ar[r] & S_k \ar[r] & C_p \ar[r] & 1 \,.
}
\]
\end{lemma}
\begin{proof}
In this proof we use additive notation. Fix the generators $\{s_i\}_{1\leq i\leq p-1}$ and the new generators $\{v_i\}_{1\leq i\leq p-1}$ defined as $v_j=\sum_{m=0}^{j-1} {\binom{j-1}{m}} s_{m+1}$, for $j = 1, \ldots, p-1$. It is a straight forward computation to check that if $A$ corresponds to an automorphism in generators $\{s_i\}_{1\leq i\leq p-1}$, $B$ corresponds to the same automorphism in generators $\{v_j\}_{1\leq j\leq p-1}$
\end{proof}

\begin{notation}\label{defiTk}
With all these computations we have obtained the following inclusion of split extensions:
\[
\xymatrix{ 1 \ar[r] & (C_{p^k})^{p-1} \ar[r] \ar@{=}[d] & S_{k} \ar[r]^{\rho_k} \ar[d] & C_p \ar[r] \ar[d] & 1 \\
1 \ar[r] & (C_{p^k})^{p-1} \ar[r] & (C_{p^k})^{p-1}\rtimes\Sigma_p \ar[r]  & \Sigma_p \ar[r] & 1
}
\]
with the action of $\Sigma_p$ as in Equation\eqref{actionsigma}, and where each term of the first row is a $p$-Sylow subgroup of the corresponding position on the second row. The kernel of $\rho_k$ plays an important role in this paper, and thus we fix the following notation: $T_k \definicio \Ker(\rho_k) \leq S_k$. We think of $T_k$ as the maximal torus of $S_k$ generated by $\gen{v_1, \ldots, v_{p-1}}$ and $s$ an element in $S_k$ of order $p$ which projects to a generator of $C_p$ and that the action of $s$ on $\gen{v_1, \ldots, v_{p-1}}$ is given by matrix $B$ in Equation~\eqref{action_B}.
\end{notation}


Below we enumerate some properties of the group $S_{k}$. To make it clear we consider the elements of $S_{k}$ written uniquely as $v_1^{i_1} \cdots v_{p-1}^{i_{p-1}}s^i$ with $0 \leq i_j \leq p^k-1$ and $0 \leq i \leq p-1$.

\begin{lemma}\label{conj-p} Consider Notation~\ref{defiTk} and let $v_1^{i_1} \cdots v_{p-1}^{i_{p-1}}$ be an element of $T_{k}\subset S_{k}$. Then, the following holds:
\begin{enumerate}
 \item $s\cdot v_1^{i_1} \cdots v_{p-1}^{i_{p-1}}\cdot s^{-1} =
v_1^{-i_{p-1}}v_2^{i_1-i_{p-1}}v_3^{i_2-i_{p-1}} \cdots v_{p-1}^{i_{p-2}-i_{p-1}}$.
 \item $v_1^{i_1} \cdots v_{p-1}^{i_{p-1}} \cdot s \cdot (v_1^{i_1} \cdots v_{p-1}^{i_{p-1}})^{-1} =
v_1^{i_{p-1} + i_1}v_2^{i_{p-1} + i_2 - i_1} \cdots v_{p-1}^{i_{p-1} + i_{p-1} - i_{p-2}} s$.
 \item $v_1^{i_1} \cdots v_{p-1}^{i_{p-1}}s$ and $v_1^{j_1} \cdots v_{p-1}^{j_{p-1}}s$ are
$S_{k}$-conjugate if and only if $\sum_{l=1}^{p-1} i_l \equiv \sum_{l=1}^{p-1} j_l \pmod{p}$.
\item The center of $S_{k}$ is cyclic of order $p$ and it is generated by $\zeta=(v_1^1v_2^2\cdots v_{p-1}^{p-1})^{p^{k-1}}$.
\item There are $p$ conjugacy classes of subgroups of order $p$ not contained in $T_{k}$, represented by the elements $v_1^is$ for $0 \leq i \leq p-1$.
\end{enumerate}
\end{lemma}
\begin{proof}
The proofs of (a) and (b) are direct computation using the action of $s$ on $v_k^{j_k}$, and statement (e) is the same as Proposition~\ref{propSpk} (b), just changing the generators. Moreover, by \cite{Blackburn}, $S_{k}$ is a $p$-group of maximal nilpotency class, which implies that the center of $S_{k}$ has order $p$. Consider $\zeta$ a generator of the center of $S_{k}$.

We prove (c) by applying (a) and (b). Consider the set $X \definicio \{ vs \}_{v \in T_{k}}$, which has $p^{k(p-1)}$ elements. By (a) and (b) there is an action of $S_{k}$ on $X$ by conjugation. This action keeps the congruence modulo $p$ of the sum of the exponents, so there are at least $p$ $S_{k}$-conjugacy classes of elements in $X$. Given an element $vs \in X$, its centralizer in $S_{k}$, $C_{S_{k}}(vs)$, is the elementary abelian group of order $p^2$ generated by $\langle vs,\zeta \rangle$: if $v'\in T_{k}\cap C_{S_{k}}(vs)$ then $v'$ commutes with all elements in $T_{k}$ and $s$, so $v \in \langle\zeta\rangle$; if $v's^i \in C_{S_{k}}(vs)$ and $v's^i\not\in T_{k}$, we can take a power of it such that $\langle v's^i\rangle=\langle v''s\rangle$, then we get $v''s\in C_{S_{k}}(vs)$, so $s^{-1}v''vs=(v''s)^{-1}(vs)=(vs)(v''s)^{-1}=v''v$, obtaining that $v''v\in \langle \zeta \rangle$. This implies that the orbits of the action of $S_{k}$ on $X$ have $p^{k(p-1)-1}$ elements. So, there are exactly $p$ $S_{k}$-conjugacy classes of elements in $X$ and the congruence modulo $p$ of the sum of the exponents determine if two of them are $S_{k}$-conjugated.
   
Next we prove the second part of (d) based on the fact that the center of $S_{k}$ has order $p$. To compute a generator of the center, consider the action of $\Sigma_p$ in the basis $\{ e_1, \dots, e_p\}$. It is easy to see that the elements of the form $\lambda (e_1 + \cdots +e_p)$ are invariant under the action of $\Sigma_p$ (in particular, by the action of $C_p$). In particular, if $\lambda$ is a multiple of $p^{k-1}$, then the corresponding element belongs to $T_{k}$. Finally, statement (e) follows by passing to the basis $\{v_1, \dots, v_{p-1}\}$ and in multiplicative notation.
\end{proof}


\subsection{Two families of exotic \texorpdfstring{$p$}{p}-local finite groups for \texorpdfstring{$p>3$}{p>3}}
We now describe some generalizations to all primes $p > 3$ of the examples constructed in Subsection~\ref{sec:finite3} for $p = 3$. We start reducing the possible outer automorphism group of $S_k$ of any saturated fusion system $\FF$ over $S_k$.
\begin{proposition}
If $(S_{k},\FF)$ is a saturated fusion system, then $\Out_\FF(S_{k})\leq C_{p-1}\times C_{p-1}$.	
\end{proposition}
\begin{proof}
As $\FF$ is saturated, $\Out_\FF(S_{k})$ must be a $p'$-subgroup of $\Out(S_{k})$.Let $\Phi(S_{k})$ be the Frattini subgroup of $S_k$. We claim that $\Phi(S_{k})=[S_{k},S_{k}]=\langle s_2, s_3, \dots s_{(p-1)k}\rangle$: as $S_{k}$ is a $p$ group, we have the inclusion $[S_{k},S_{k}]\subset\Phi(S_{k})$, and, in this case, $S_{k}/[S_{k},S_{k}]$ is elementary abelian of rank $2$; if we had that $[S_{k},S_{k}]\neq \Phi(S_{k})$, then $S_{k}/\Phi(S_{k})$ would have rank $1$, which would imply that $S_{k}$ is cyclic. The kernel of the map
	\[
	\rho\colon\Out(S_{k}) \to \Out(S_{k}/\Phi(S_{k}))
	\]
	is a $p$-group, so $\rho(\Out_\FF(S_{k}))$ is isomorphic to a subgroup of $\Out(S_{k}/\Phi(S_{k}))$. Use now that $S_{k}/\Phi(S_{k})$ is a rank two elementary abelian group, so we can consider as an $\F_p$ vector space with basis $\{\overline{s},\overline{s_1}\}$ (images of $s$ and $s_1$ in $S_{k}/\Phi(S_{k})$). We have the equality $\gamma=C_{S_{k}}([S_{k},S_{k}])$, so $\gamma$ is a characteristic subgroup of $S_{k}$. This implies that in this basis, $\rho(\Out_\FF(S_{k}))$ is included in lower triangular matrices of $\GL_2(p)$. Now use again that $p\nmid \#\Out_\FF(S_{k})$ to get the result.
\end{proof}

\begin{notation}\label{not:phipsi}
There are some subgroups of $S_k$ which will be of special interest in this subsection, and whose notation we now recall or fix. Let $T_k \leq S_k$ be the maximal torus of $S_k$ and the element $s\in S_k$ as defined in Notation~\ref{defiTk}. Consider now
$$
V \definicio \gen{\zeta, s} \leq S_k,
$$
where $\zeta \in Z(S_k)$ is the generator of the center of $S_k$ specified in Lemma~\ref{conj-p} (d). Thus, $V$ is an elementary abelian $p$-subgroup of rank $2$. As we work with concrete examples, we also specify certain automorphisms, $\phi, \psi \in \Aut(S_k)$ of order $p-1$, for later use:
\begin{itemize}
\item The normalizer of $\langle s \rangle$ in $\{1\}\rtimes \Sigma_p$ is isomorphic to $C_p\rtimes C_{p-1}$. Let $\phi \in \{1\}\rtimes\Sigma_p$ be an element of order $p-1$ normalizing $\langle s  \rangle$. $\phi$ acts over $S_{k}=T_{k}\rtimes \langle s\rangle$ by conjugation (here we consider $S_{k}$ as a subgroup of $T_{k}\rtimes \Sigma_p$). This action sends $s\mapsto s^\lambda$, $\lambda$ a generator of $\F_p^\times$, while $\phi(\zeta)=\zeta$.
\item As $\Aut(C_{p^k})\cong (\Z/p^k\Z)^\times$, we can consider $\mu \in \Aut(C_{p^k})$ an element of order $p-1$, and define $\psi$ as the element in $\Aut(S_{k})$ which restricts to $\mu\times\stackrel{p-1}{\cdots}\times\mu$ in the maximal torus and to the identity on $\langle s \rangle$. To facilitate computations, we can choose $\mu$ such that the composition $\Aut(C_{p^k})\cong(\Z/p^k\Z)^\times \to (\Z/p\Z)^\times$ (where the last map is the reduction modulo $p$) sends $\mu$ to $\lambda$. With this definition, $\psi(\zeta)=\zeta^\lambda$.
\end{itemize}
It can be checked that $\langle \phi, \psi\rangle \cong C_{p-1}\times C_{p-1}$. Moreover, restrictions to $V$ and $T_k$ are given by $\phi|_V=\big(\begin{smallmatrix}\lambda & 0\\ 0 & 1 \end{smallmatrix} \big)$, $\psi|_V=\big(\begin{smallmatrix}1 & 0 \\ 0 & \lambda \end{smallmatrix} \big)$ (we are using additive notation and taking $V=\langle s,\zeta\rangle$), and by $\phi|_{T_{k}}=\sigma$ ($\sigma\in\Sigma_p$ an element of order $(p-1)$ normalizing $s$) and $\psi|_{T_{k}}=\lambda\Id$.
\end{notation}

\begin{proposition}\label{def_of_sfs}
Consider the subgroups $T_{k}$ and $V$ of $S_k$, as given in Notation~\ref{defiTk} and \ref{not:phipsi} respectively. Consider also the action of $\Sigma_p$ on $T_{k}$ described in Equation~\eqref{actionsigma} and let $\phi$ and $\psi$ be the automorphisms of $S_k$ fixed in  Notation~\ref{not:phipsi}.
\begin{enumerate}
\item  For $p \geq 5$, there are exotic saturated fusion systems $(S_{k},\widetilde\FF_{k})$  generated, in the sense of Remark~\ref{generated}, by the following automorphisms groups:
\begin{itemize}
	\item $\Aut_{\widetilde{\FF}_k}(S_k)=\langle \phi,\psi,\Inn(S_k) \rangle$ with an isomorphism $\Out_{\widetilde{\FF}_k}(S_k) \cong C_{p-1}\times C_{p-1}$,
	\item $\Aut_{\widetilde{\FF}_k}(T_{k})=\langle \Sigma_p,\psi\rangle \cong \Sigma_p\times C_{p-1}$ and
	\item $\Aut_{\widetilde{\FF}_k}(V)=\GL_2(\F_p)\cong \SL_2(\F_p)\rtimes C_{p-1}$.
\end{itemize}
\item $(S_k,\widetilde{\FF}_k)$  contains an exotic simple saturated fusion subsystem $(S_k,\FF_k)$ of index $2$, generated by the following automorphisms groups:
\begin{itemize}
	\item $\Aut_{\FF_k}(S_k)=\langle \phi^2,\psi\phi^{-1},\Inn(S_k) \rangle$ with an isomorphism $\Out_{\FF_k}(S_k) \cong C_{\frac{p-1}{2}}\times C_{p-1}$,
	\item $\Aut_{\FF_k}(T_{k})=\langle A_p , \psi\phi^{-1}\rangle\cong A_p\rtimes C_{p-1}$ (where $A_p\leq\Sigma_p$ is the alternating group) and
	\item $\Aut_{\FF_k}(V)=\SL_2(\F_p)\rtimes C_{\frac{p-1}{2}} < \GL_2(\F_p)$.
\end{itemize}
\end{enumerate}
Moreover, $S_k$, $T_k$ and $V$ are representatives of the only conjugacy classes in $\widetilde{\FF}_k$ (respectively $\FF_k$) of centric radical subgroups of $S_k$.
\end{proposition}
\begin{proof}
To prove (a), the existence of the saturated fusion systems $(S_k,\widetilde{\FF}_k)$ can be found in \cite[Example 9.3]{BLO2}. In the same result the authors also prove that these examples are exotic.

To get (b), we can proceed classifying all the saturated fusion subsystems of $(S_k,\widetilde{\FF}_k)$ of index prime to $p$ as in \cite[Section 5.1]{BCGLO2} or \cite[Part I.7]{AKO}. To do this, we need to compute $\EE_0$, the fusion system generated by $O^{p'}(\Aut_{\widetilde{\FF}_k}(P))$ for all $P\in\widetilde{\FF}_k^{c}$ and use it to compute
$$
\Aut^0_{\widetilde{\FF}_k}(S_k)=\langle \alpha\in\Aut_{\widetilde{\FF}_k}(S_k)\mid \alpha|_p \in \Hom_{\EE_0}(P,S_k), \text{ some } P \in \FF_k^c\rangle .
$$
By \cite[Theorem 3.4]{Ruiz}, it is enough to describe the groups $O^{p'}(\Aut_{\widetilde{\FF}_k}(Q))$, where $Q$ is centric and radical in $\widetilde{\FF}_k$.
\begin{itemize}
	\item For $P=S_k$, $O^{p'}(\Aut_{\widetilde{\FF}_k}(S_k))=\Inn(S_k)$.
	\item For $P=T_{k}$, $O^{p'}(\Aut_{\widetilde{\FF}_k}(T_{k}))\cong A_p$: the elements of order $p$ in $\Sigma_p$ for odd prime $p$ generate the alternating group $A_p$.
	\item For $P=V$, $O^{p'}(\Aut_{\widetilde{\FF}_k}(V))\cong\SL_2(\F_p)$: for odd prime $p$ the elements of order $p$ in $\GL_2(\F_p)$ generate $\SL_2(\F_p)$.
\end{itemize}
Now we have to detect the elements in $\Out_{\widetilde{\FF}_k}(S_k)$ which restrict to morphisms in $\Out_{\EE_0}(S_k)$. We recall that, by definition, there is $\lambda$ a generator of $\F_p^\times$ such that $\phi|_V=\big(\begin{smallmatrix}\lambda & 0\\ 0 & 1 \end{smallmatrix} \big)$, $\psi|_V=\big(\begin{smallmatrix}1 & 0 \\ 0 & \lambda \end{smallmatrix} \big)$, as matrices of $\GL_2(\F_p)$ (see Notation \ref{not:phipsi} for details).
\begin{itemize}
\item $\phi^2$ is an even permutation, so it restricts to an element in 
$O^{p'}(\Aut_{\widetilde{\FF}_k}(T_{k}))$.
\item $\psi^i \phi^{-i}$ restricts to an automorphism of determinant one in $V$.
\item $\langle \phi^2, \psi\phi^{-1}\rangle$ is a subgroup of index $2$ in $\langle \phi,\psi \rangle$ and it remains to prove that $\phi \not\in \Aut_{\EE_0}(S)$. Using \cite[Theorem 3.4]{Ruiz} it is enough to check that $\phi\not\in O^{p'}(\Aut_{\widetilde{\FF}_k}(P))$ for $P=V$ and $P=T_{k}$, which are the only proper $\widetilde{\FF}_k$-centric, $\widetilde{\FF}_k$-radical subgroups (remark that there are not inclusions between them): $\phi$ is an odd permutation, so the restriction to $T_{k}$ does not give an element of $O^{p'}(\Aut_{\widetilde{\FF}_k}(T_{k}))$; $\phi$ does not restrict to an automorphism of determinant $1$ in $V$, so $\phi \not \in O^{p'}(\Aut_{\widetilde{\FF}_k}(V))$. 
\end{itemize}
These computations show us that $\Aut_{\widetilde{\FF}_k}(S)/\Aut_{\widetilde{\FF}_k}^0(S)\cong\Z/2\Z$, so there is just one proper saturated fusion subsystem $\FF'\definicio O^{p'}(\widetilde{\FF}_k)$ of index prime to $p$, and it is of index $2$. By these computations, $\langle \phi^2, \psi\phi^{-1}\rangle \leq \Aut_{\FF'}(S_k)$, $\SL_2(\F_p)\leq\Aut_{\FF'}(V)$ and $A_p\leq \Aut_{\FF'}(T_k)$. 

Consider $\FF_k$ the (not necessarily saturated) fusion system generated, in the sense of Remark~\ref{generated}, by the automorphims in part (b) of the theorem. The restriction $\phi^2|_V$ must also be in $\Aut_{\FF'}(V)$, proving that $\Aut_{\FF_k}(V)\leq \Aut_{\FF'}(V)$, and the restriction $\phi\psi^{-1}|_{T_k}$ must also be in $\Aut_{\FF'}(T_k)$, getting that $\Aut_{\FF_k}(Q) \leq \Aut_{\FF'}(Q)$ for all $Q\in\{V,T_k,S_k\}$. 

It remains to check that there are no elements in $\Aut_{\FF'}(Q)\setminus \Aut_{\FF_k}(Q)$ for $Q\in\{V,T_k,S_k\}$: in these three cases $\Aut_{\widetilde{\FF}_k}(Q)/\Aut_{\FF_k}(Q)\cong \Z/2\Z$, so adding any other morphism $\alpha \in \Aut_{\widetilde{\FF}_k}(Q)\setminus\Aut_{\FF_k}(Q)$ implies $\Aut_{\FF'}(Q_0)=\Aut_{\widetilde{\FF}_k}(Q_0)$ with $Q_0$ either $V$, $T_k$ or $S_k$. If $Q_0=S_k$, $\alpha$ would restrict to elements in $\Aut_{\FF'}(Q)$ not in $\Aut_{\FF_k}(Q)$ for $Q\in\{V,T_k\}$, obtaining $\FF'=\widetilde{\FF}_k$. If $Q_0\in\{V,T_k\}$, then $\Aut_{\FF'}(Q_0)=\Aut_{\widetilde{\FF}_k}(Q_0)$ and, using that $\FF'$ is saturated, there is an automorphims which would extend to $\psi\in\Aut_{\FF'}(S_k)$, getting that also in this case $\FF'=\widetilde{\FF}_k$. As $\FF'\lneq\widetilde{\FF}_k$, we get that $\Aut_{\FF'}(Q)=\Aut_{\FF_k}(Q)$ for $Q\in\{V,T_k,S_k\}$ and $\FF'=\FF_k$.

Let us see now that $\FF_k$ is simple. By Lemma~\ref{reducedsimple} we have to check that:
\begin{enumerate}[label=\textup{(\roman*)}]
\item $O_p(\FF_{k})=1$: there is not any proper nontrivial strongly closed subgroup in $\FF_{k}$. Indeed, let $P \trianglelefteq S_{k}$ be a 
nontrivial strongly closed subgroup. By \cite[Theorem 8.1]{AB}, $P$ must intersect the center in a nontrivial subgroup. Since the center of $S_{k}$ 
has order $p$, we must have $Z(S_{k}) \leq P$. Moreover, since $\zeta$ is $\FF_{k}$-conjugated to $s$ by a morphism in $\Aut_{\FF_{k}}(V)$, we 
must have also $s \in P$, since $P$ is strongly closed. Then, by \cite[Lemma 2.2]{Blackburn}, $P$ is of index at most $p$ in $S_{k}$.

In fact, by Lemma~\ref{conj-p}~$(c)$, $P$ must contain the subgroup generated by $s$ and all elements $v_1^{i_i}\cdots v_{p-1}^{i_{p-1}} \in 
T_{k}$ whose sum of exponents is congruent to 0 modulo $p$, which is an index $p$ subgroup of $S_{k}$. Then, let $\varphi$ be the automorphism of $T_{k}$ induced by the cycle $(1 2 3) \in A_p \cong \Aut_{\FF_{k}}(T_{k})$. We have that
\begin{align*}
\varphi(v_1) &= v_2\\
\varphi(v_2) &= v_1^{-1}v_2^{-1}\\
\varphi(v_3) &= v_1v_2v_3\\
\varphi(v_i) &= v_i, \text{ for } 4 \leq i \leq p-1
\end{align*}
Taking for example the element $v_2^{-1}v_4$, which is in $P$ since the sum of exponents is 0, we have that $\varphi(v_2^{-1}v_4) = v_1v_2v_4$. 
Thus, $v_1v_2v_4$ must lie also in $P$, but since the sum of the exponents is not 0 modulo $p$ for $p \geq 5$, and using that the index of $P$ is at 
most $p$, we get that $P = S_{k}$. So there is not any proper nontrivial normal subgroup in $\FF_{k}$.
\item $O^p(\FF_k)=\FF_k$: by \cite[Proposition 1.3 (d)]{Oliver_Ind_p}, we need to show that $\foc(\FF_k)=S_k$, where $\foc(\FF_k)$ is the \emph{focal subgroup of $\FF_k$}, defined by
\[
\foc(\FF_k) = \gen{g \cdot \varphi(g)^{-1} \, |\, g \in S_k, \, \varphi \in \Hom_{\FF_k}(\gen{g}, S_k)}
\]
(remark that this definition applies only to finite saturated fusion systems).
First, note that there are elements $\varphi,\varphi'\in\Aut_{\FF_k}(V)$ such that $\varphi(s)=s \zeta$ and also $\varphi'(\zeta)=s\zeta$, and hence we get $V_0\subset\foc(\FF_k)$. The action of $C_{p-1}$ on the maximal torus $T_k$ includes $\varphi$ such that $\varphi(v)=v^{-1}$ for all $v \in T_k$, so we have all elements $\langle v_1^2, \dots , v_{p-1}^2\rangle\subset\foc(\FF_k)$. Taking now the expression of $\zeta$, we get that $v_1 v_3 \cdots v_{p-2} \in \foc(\FF_k)$. Conjugating this element by $s$ we get $v_2 v_4 \cdots v_{p-1} \in \foc(\FF_k)$. Consider now $\varphi$ the conjugation by an element of order $p$ in $A_p$ on $v_{p-1}$. This tells us that $v_{p-1}\varphi(v_{p-1}^{-1})=v_1 v_2 \dots v_{p-2}v_{p-1}^2\in\foc(\FF_k)$. So $v_1 v_2 \dots v_{p-2}\in\foc(\FF_k)$ and
we get that $v_{p-1}\in\foc(\FF_k)$. Conjugating $v_1 v_2 \dots v_{p-2}$ by $s$ and using that $v_{p-1}\in\foc(\FF_k)$ we get $v_1\in\foc(\FF_k)$, getting that $S_k=\langle s,v_1 \rangle \subset\foc(\FF_k)$.
\item $O^{p'}(\FF_k)=\FF_k$: as $\widetilde{\FF}_k^c=\FF_k^c$, the computations in the first part of this proof show us that $O^{p'}(\FF_k)=O^{p'}(\widetilde\FF_k)=\FF_k$.
\end{enumerate}
Finally, $(S_k,\FF_k)$ and $(S_k,\widetilde{\FF}_k)$ are exotic because these examples do not appear in the list of \cite[Proposition 9.5]{BLO2}, which contains all the finite groups realizing simple saturated fusion systems over $p$-groups of this type.
\end{proof}

\section{New \texorpdfstring{$p$}{p}-local compact groups}\label{sec:exotic}

In this section we will introduce the $p$-local compact version of the finite examples described in Section~\ref{sec:finite} and we will discuss about their exoticness. When trying to generalize the definition of \emph{exotic} from $p$-local finite groups to to $p$-local compact groups, it seems natural to just remove the finiteness condition but, as we can see in \cite{GL}, there always exists some (non-compact) infinite group which realizes a given saturated fusion system over a discrete $p$-toral group. So, in order to keep the condition of being compact, we restrict our attention to compact Lie groups and $p$-compact groups.
%
\begin{remark}
There are several ways of producing examples of \emph{non-simple} $p$-local compact groups which are not $p$-compact groups. For example, one could consider an extension of a torus by a non $p$-nilpotent finite group. By Lemma~\ref{lemma_ishiguro}, this produces a $p$-local compact group which does not correspond to any $p$-compact group.
\end{remark}

\subsection{Families of \texorpdfstring{$p$}{p}-local compact groups}\label{sec:plcg}

In this subsection we describe the properties which define the $p$-local compact groups we are interested in. 
\begin{notation}\label{notation_ptoral}
Consider the extension $T_{k} \rightarrow S_{k} \rightarrow C_p$ and the particular elements $s$ and $\{v_i\}_{1\leq i \leq p-1}$ described in Notation~\ref{defiTk}. We can construct a discrete $p$-toral group by taking the monomorphims $I_{k}\colon S_{k} \rightarrow S_{k+1}$, defined by $I_{k}(s) = s$ and $I_{k}(v_i) = v_i^p$, which are compatible with an obvious choice of sections of the extensions. Thus, the discrete $p$-toral group
\[S = \bigcup_{k \geq 2} S_{k}\]
fits in a split extension $T \rightarrow S \rightarrow C_p$, where $T = \bigcup_{k \geq 2} T_k \cong (\Z/p^{\infty})^{p-1}$. Observe that, for each $k \geq 2$, the generator $\zeta$ of $Z(S_k)$ described in Lemma~\ref{conj-p} (d) is mapped by $I_k$ to the corresponding generator of $Z(S_{k+1})$. It is thus reasonable to adopt the same notation for the resulting element in $S$, and in fact we have $Z(S) = \gen{\zeta}\cong \Z/p$. We also consider $V = \gen{s, \zeta} \cong \Z/p \times \Z/p$ as a subgroup of $S$, via the obvious inclusion. Finally, notice that the description of the different subgroups of $\Out(P)$ for $P\in\{S_{k},T_{k},V\}$ described in Section~\ref{sec:finite} can be generalized without modification to describe certain subgroups of $\Out(Q)$, for $Q \in \{S, T, V\}$. For later use, we also set $Z \definicio Z(S)$.
\end{notation}
\begin{theorem}\label{FFpissaturated}
Let $p$ be an odd prime number and $S$ as in Notation~\ref{notation_ptoral}.
Consider $(S,\FF)$ for $p\geq 3$,  and $(S,\widetilde{\FF})$ for $p\geq 5$, the fusion systems generated, in the sense of Remark~\ref{generated}, by the automorphisms in the following table:
	{\renewcommand{\arraystretch}{2}
	\[
	\begin{array}{|c||c|c|c|c|}
	\hline \FF & \Aut_\FF(S) & \Aut_\FF(T) & \Aut_\FF(V) & \text{prime} \\
	\hline \hline \FF & \langle \phi, \psi, \Inn(S) \rangle  & \GL_2(\F_3) & \GL_2(\F_3) & p=3 \\
	\hline \FF &  \langle \phi^2, \psi\phi^{-1}, \Inn(S) \rangle & A_p \rtimes C_{p-1} & \SL_2(\F_p) \rtimes C_{(p-1)/2} & \multirow{2}{30pt}{$p\geq5$} \\
	\cline{1-4} \widetilde{\FF} &  \langle \phi, \psi, \Inn(S) \rangle & \Sigma_p \times C_{p-1} & \GL_2(\F_p) & \\ \hline
	\end{array}
	\]
	}
where $\phi$ and $\psi$ are defined in Notation~\ref{not:phipsi} and the action of the symmetric group on maximal torus is the one described in Equation~\eqref{actionsigma}. Then, $\FF$ (for $p\geq3$) and $\widetilde{\FF}$ (for $p\geq5$) are saturated fusion systems over $S$ and there exist $p$-local compact groups $(S,\FF,\LL)$ (for $p\geq3$) and $(S,\widetilde\FF,\widetilde\LL)$ (for $p\geq5$). Moreover, $S$, $T$ and $V$ are representatives of the only conjugacy classes in $\widetilde{\FF}$ (respectively $\FF$) of centric radical subgroups of $S$.
\end{theorem}


The proof of Theorem~\ref{FFpissaturated} essentially relies upon Proposition~\ref{1.1LO}. First, we need some technical lemmas.



\begin{lemma}\label{conjtotor}
Let $(S, \EE)$ be any of the fusion systems described in Theorem~\ref{FFpissaturated}, and let $x \in S$ be an element of order $p$. Then, either $x \in T$ or $x$ is $\EE$-conjugate to an element of $Z$.
\end{lemma}

\begin{proof}
Let $v\in T$ and $vs^i$, with $i \neq 0$, an element not contained in $T$. An easy computation shows that
$\langle vs^i\rangle=\langle ws\rangle$ for some $w \in T$. Since $T = (\Z/p^{\infty})^{p-1}$, we can write $w
= w_1^{pi_1} \cdots w_{p-1}^{pi_{p-1}}$, for some $w_1, \dots, w_{p-1}$ in $\Z/p^k\subset\Z/p^\infty$, for $k$ big enough. Then, by Lemma~\ref{conj-p} (c), $w_1^{pi_1} \cdots w_{p-1}^{pi_{p-1}}s$ and $s$ are conjugate in $S_{k}$, and thus so are they in $S$. Hence $vs^i$ is $S$-conjugate to $s^j$ for some $j \neq  0$. Finally, as $\FF \subseteq \widetilde{\FF}$ for $p \geq 5$, and as
\[
\Out_{\FF}(V) = \left\{
\begin{array}{ll}
\GL_2(\F_3), & \mbox{if $p = 3$;}\\
\SL_2(\F_p) \rtimes C_{\frac{p-1}{2}}, & \mbox{if $p \geq 5$;}
\end{array}
\right.
\]
we deduce that $s$ is conjugate to $s^j$ for all $j \neq 0$ in all cases. This also proves that $s$ is conjugate to the element $\zeta \in Z$, and this finishes the proof.
\end{proof}

The following result states some properties of centralizers of toral elements of order $p$. In the particular case of central elements, we describe the centralizer fusion system in full detail, as we will need these computations later in this section.

\begin{lemma}\label{centralZ}
Let $(S, \EE)$ be any of the fusion systems described in Theorem~\ref{FFpissaturated}, let $v \in T$ be a nontrivial element of order $p$, and let $C_{\EE}(v)$ be the centralizer fusion system of $\gen{v}$ over $C_S(v)$. Then, $C_{\EE}(v)$ is a saturated fusion system over $C_S(v)$. Moreover, if $v \in Z$, then $C_{\EE}(v)$ is given by the following table.
	{\renewcommand{\arraystretch}{2}
	\[
	\begin{array}{|c|c|}
	\hline p = 3 & C_{\FF}(v) = \FF_{S}(T \rtimes \Sigma_3)\\
    \hline	\multirow{2}{30pt}{$p \geq 5$} & C_{\FF}(v) = \FF_{S}(T \rtimes A_p)\\
	\cline{2-2} & C_{\widetilde{\FF}}(v) = \FF_{S}(T \rtimes \Sigma_p) \\
 \hline
	\end{array}
	\]
	}
\end{lemma}

\begin{proof}
Fix $v \in T$, a nontrivial element of order $p$, and let $W = \gen{v}$. By definition, $C_{\EE}(v)$ is the fusion system over $C_S(v)$ whose morphisms are those morphisms $\alpha \colon P \to P'$ in $\EE$ that extend to some $\widetilde{\alpha} \colon PW \to P'W$ such that $\widetilde{\alpha}|_W = \Id_W$.

Suppose first the following: for every morphism $\alpha \colon P \to P'$ in $C_{\EE}(W)$, with $W \leq P$, there is a factorization
$$
\alpha = \alpha_3|_{X_3} \circ \alpha_2|_{X_2} \circ \alpha_1|_P,
$$
with $X_2 = \alpha_1(P)$ and $X_3 = \alpha_2(X_2)$, satisfying the following conditions:
\begin{enumerate}[label=\textup{(\arabic*)}]
\item $\alpha_1, \alpha_3 \in \Aut_{\EE}(S)$ and $\alpha_2 \in \Aut_{\EE}(Q)$, with $Q \in \{S, T, V\}$; and
\item $\alpha_i|_W = \Id_W$ for $i = 1, 2, 3$ (in particular, $W \leq Q$ in condition (1)).
\end{enumerate}
Under this assumption, we claim that the statement follows: assume that $v \notin Z$, so that $C_S(v) = T$ and $W \not \leq V$. This means that every morphism in $C_{\EE}(v)$ is the composition of restrictions of automorphisms of $S$ or $T$ which restricts to the identity in $\langle v\rangle$. As $C_S(v) = T$, we get that $C_{\EE}(v)=\FF_{T}(G)$, with $G\cong T\rtimes W$ ($W$ a finite group). Such a $G$ satisfies the conditions in \cite[Theorem 8.7]{BLO3}, getting that $C_{\EE}(v)$ is saturated.
Suppose, otherwise, that $v \in Z$, so that $W = Z$ and $C_S(v) = S$. By properties (1) and (2) above, in order to describe $C_{\EE}(v)$, it is enough to analyze the groups $\Gamma_Q\definicio\{\gamma \in \Aut_{\EE}(Q) \mid \gamma|_Z = \Id_Z\}$ for $Q \in \{S, T, V\}$. A case by case inspection shows that $\Gamma_Q=\Aut_Q(G)$, where $Q \in \{S, T, V\}$ and $G$ is the group specified in the statement. So $C_{\EE}(v)$ corresponds to the fusion system $\FF_{S}(G)$, with $G$, again, as in \cite[Theorem 8.7]{BLO3}, which implies that $C_{\EE}(v)$ is saturated.

It remains to prove that each morphism in $C_{\EE}(v)$ admits a factorization satisfying properties (1) and (2) above. Thus, fix some morphism $\alpha \colon P \to P'$ in $C_{\EE}(v)$. Without loss of generality we may assume that $W \leq P, P'$. As $\EE$ is generated by the automorphisms of $S$, $T$ and $V$ specified in Theorem~\ref{FFpissaturated}, there exist $\beta_i \in \Aut_\EE(Q_i)$, with $Q_i \in \{V, T, S\}$ such that
\begin{equation}\label{alphabeta}
\alpha=\beta_n|_{P_n} \circ \cdots \circ \beta_1|_{P_1}
\end{equation}
where $P_1=P$ and $P_i=\beta_{i-1}(P_{i-1})$. Consider now $n_0$ minimal such that a factorization of $\alpha$ exists. Such a minimal factorization fulfills the following conditions.
\begin{enumerate}
\item There is not $P_i=P_{i+1}$, as we can always compose $\alpha_i$ with $\alpha_{i+1}$ to obtain a shorter factorization.
\item There is only one index $i_0$ such that $P_{i_0}=V$: if there are two indices $i_0<i_1$ with $P_{i_0} = P_{i_1} = V$, then the composition $\beta_{i_1}|_{P_{i_1}} \circ\cdots\circ \beta_{i_0}|_{P_{i_0}}$ is the restriction of an element in the group $\Aut_{\EE}(V) \geq \SL_2(\F_p)$.
\item If $P_i=T$, then $P_{i-1}\neq V$ and $P_{i+1}\neq V$: as $T\cap V=Z$, this composition can only restrict to an automorphism of $Z$, and these are all realized by restrictions of elements of $\Aut_\EE(S_p)$.
\item There is only one index $i_0$ such that $P_{i_0}=T$: by the previous cases, if there are two indices $i_0<i_1$ such that $P_{i_0}=P_{i_1}=T$, then the above conditions imply that $i_1=i_0+2$ and $Q_{i_0+1}=S_p$. As $T$ is a characteristic subgroup of $S_p$, the composition $\beta_{i_0+2}\circ \beta_{i_0+1} \circ \beta_{i_0}$ can be realized as single element in $\Aut_\EE(T)$, producing again a shorter factorization.
\end{enumerate}
With these restrictions, it follows that there always exists a factorization $\alpha = \beta_3|_{P_3} \circ \beta_2|_{P_2} \circ \beta_1|_P$, where $\alpha_1, \alpha_3 \in \Aut_{\EE}(S)$ and $\alpha_2 \in \Aut_{\EE}(Q)$ for $Q \in \{V, T, S\}$. Notice that such factorization may not be of minimal length, although that means no inconvenience.

Next, we refine such a factorization of $\alpha$, by modifying the morphisms $\beta_i$ if necessary, so that each morphism in the factorization restricts to the identity on $W$. To do that, let $\Aut_T(S) \leq \Aut_{\EE}(S)$ be the subgroup of automorphisms induced by conjugation by elements of $T$, let $H \leq \Aut_{\EE}(S)$ be the subgroup of automorphisms induced by conjugation by a power of the element $s \in S$, and let $G \leq \Aut_{\EE}(S)$ be the subgroup $\gen{\phi, \psi}$ or $\gen{\phi^2, \psi\phi^{-1}}$, according to the appropriate case in Theorem~\ref{FFpissaturated}. This way, $\Inn(S) = \gen{\Aut_T(S), H}$, and $\Aut_{\EE}(S) = \gen{G, \Inn(S)}$. Moreover, as $T$ is characteristic in $S$, $\Aut_T(S)$ is a normal subgroup of $\Aut_{\EE}(S)$.

Let $\alpha = \beta_3|_{P_3} \circ \beta_2|_{P_2} \circ \beta_1|_P$ be as above. By the above discussion, the automorphisms $\beta_1, \beta_3 \in \Aut_{\EE}(S)$ admit factorizations
\[
\beta_1 = \gamma_3 \circ \gamma_2 \circ \gamma_1 \qquad \mbox{and} \qquad \beta_3 = \lambda_1 \circ \lambda_2 \circ \lambda_3
\]
where $\gamma_3, \lambda_3 \in G$, $\gamma_2, \lambda_2 \in H$ and $\gamma_1, \lambda_1 \in \Aut_T(S)$. Notice that all the elements of $G$ and $H$ restrict to automorphisms of both $V$ and $T$. Hence, the morphism $\alpha$ can be factored as
\[
\alpha = \alpha_3|_{Q_3} \circ \alpha_2|_{Q_2} \circ \alpha_1|_P,
\]
where $\alpha_1 = \gamma_1$, $\alpha_2 = \lambda_2 \circ \lambda_3 \circ \beta_2 \circ \gamma_3 \circ \gamma_2$ and $\alpha_3 = \lambda_1$, and where $Q_2 = \alpha_1(P)$ and $Q_3 = \alpha_2(Q_2)$. Here, the morphisms $\lambda_i, \gamma_i$ for $i = 2, 3$ are restricted to match the domain of $\beta_2$ in order to produce $\alpha_2$. In particular, the new factorization of $\alpha$ still satisfies that $\alpha_1, \alpha_2 \in \Aut_{\EE}(S)$ and $\alpha_2 \in \Aut_{\EE}(Q)$, with $Q \in \{V, T, S\}$. Moreover, as $W \leq T$, we have $\beta_1|_W = \beta_3|_W = \Id_W$, and $\alpha|_W = \Id_W$ by assumption, which implies that $\beta_2|_W = \Id_W$. Notice also that $W \leq Q$.
\end{proof}

\begin{proof}[Proof of Theorem~\ref{FFpissaturated}]
Let $(S, \EE)$ be any of the fusion systems described in Theorem~\ref{FFpissaturated}. In order to show that $\EE$ is saturated, we apply Proposition~\ref{1.1LO}. Note that, by construction, $\EE$ satisfies axiom (III) of saturated fusion systems, which corresponds to condition (a) in Proposition~\ref{1.1LO}, and it remains to check condition (b): there exists a set $\setp$ of elements of order $p$ in $S$ satisfying conditions (i) through (iii) in Proposition~\ref{1.1LO}.

Let  $\setp = \lbrace v \in T \mid v \mbox{ is of order } p \rbrace$. Then, condition (i) in Proposition~\ref{1.1LO} follows from Lemma~\ref{conjtotor}, and condition (iii) in Proposition~\ref{1.1LO} follows from Lemma~\ref{centralZ}. It remains to check condition (ii) in Proposition~\ref{1.1LO}: if $x, y$ are $\EE$-conjugate and $y \in \setp$, there is some morphism $\rho \in \Hom_{\EE}(C_S(x), C_S(y))$ such that $\rho(x) = y$. Observe that, for $v \in T$,  $C_S(v) = T$ if $v \not\in Z$ and $C_S(v) = S$ if $v \in Z$.

Let first $v \in T$ be a non-central element of order $p$. If $v$ is $\EE$-conjugated to other $v' \in T$, then, by construction, there is an automorphism $\rho \in \Aut_{\EE}(T)$ such that $\rho(v) = v'$. If $v \in Z$, then $v = \zeta^\lambda$, for some $\lambda$, and $\zeta^\lambda$ is conjugated to $\zeta^\mu$ for all $\lambda, \mu \neq 0$ by an $\EE$-automorphism of $S$.

Finally, let $vs^i$ be an element of order $p$, with $v \in T$ and $i \neq 0$. By Lemma~\ref{conj-p} (b) there exists some $t \in T$ such that $t(vs^i)t^{-1} = s^i$. Hence, it is enough to prove that there is some $\rho \in \Hom_{\EE}(C_{S}(s^i), C_{S}(\zeta^\lambda))$ such that $\rho(s^i) = \zeta^\lambda$ for some $\lambda \neq 0$. Recall that $C_{S}(s^i) = V$ and $C_{S}(\zeta^\lambda)) = S$. Moreover, by construction there is an automorphism $\rho \in \Aut_{\EE}(V)$ sending $s^i$ to $\zeta^\lambda$, and its composition with the inclusion $V \leq S$ yields a morphism
\[
\rho \in \Hom_{\EE}(C_{S}(s^i), C_{S}(\zeta^\lambda))
\]
such that $\rho(s^i) = \zeta^{\lambda}$. This proves that condition (ii) in Propositon~\ref{1.1LO} is satisfied, and thus $\EE$ is saturated. The existence and uniqueness of a linking systems for every saturated fusion system is proved by R.~Levi and A.~Libman in \cite{LeviLibman}.
\end{proof}

The next step is to prove that our examples have no proper nontrivial strongly closed subgroups.This will be very useful to study the normal 
subsystems.
\begin{lemma}\label{no-proper-nontrivial-strongly-closed-subgroups}
Let $(S, \EE)$ be any of the fusion systems described in Theorem~\ref{FFpissaturated}. Then, $\EE$ contains no proper nontrivial strongly closed subgroups.
\end{lemma}

\begin{proof}
Let $P \leq S$ be a strongly closed nontrivial subgroup. In particular, $P$ is normal in $S$ and, if we write $P_k = S_{k} \cap P$, we 
have that $P_k$ is normal in $S_{k}$. Since $P$ is nontrivial, there exists $k$ such that $P_k$ is also nontrivial and, by \cite[Theorem 
8.1]{AB}, the center $Z(S_{k})$ intersects $P_k$ in a non trivially way. Since the center has order $p$, we must have $Z(S_{k}) \leq P_k$. This 
implies that $Z = Z(S) \leq P$.

The generator of the center is $\EE$-conjugated to $s$, as proved in Lemma~\ref{conjtotor}. Hence, $s \in P$, since $P$ is strongly closed by assumption. 
Moreover, we saw in the proof of Lemma~\ref{conjtotor} that all elements not in the maximal torus are conjugated to $s$, so all elements not in the 
maximal torus must belong also to $P$. Finally, if $P$ contains $s$ and all the elements not in the maximal torus it also contains all elements of the maximal torus, so the only 
possibility is $P = S$, and then $P$ is not proper.
\end{proof}

We finish this section with some properties of the examples described in Theorem~\ref{FFpissaturated}.
\begin{proposition}\label{prop_FFp_simple}
The following holds.
\begin{enumerate}
\item For $p \geq 3$, the fusion system $(S, \FF)$ is simple.
\item For $p \geq 5$, $(S, \FF)$ is the only proper normal subsystem of $(S, \widetilde{\FF})$, with index $2$. 
\item Let $(S, \EE, \TT)$ be any of the $p$-local compact groups described in Theorem~\ref{FFpissaturated}. Then, $\pi_1(|\TT|^\wedge_p)=\{e\}$.
\end{enumerate}
\end{proposition}
\begin{proof}
To prove (a) we have to check that every proper normal subsystem of $(S,\FF)$ is finite. As there is not any proper $\FF$-strongly closed subgroup in $S$, we only have to check that any normal subsystem over $S$ must be all $\FF$. Let $\FF' \subseteq \FF$ be a normal subsystem over $S$. By condition (N2) in Definition~\ref{definormal}, as we have to consider fusion subsystems over $S$, all the elements of order $p$ in $\Out_{\FF}(P)$ must be in $\Out_{\FF'}(P)$. This implies that:
\begin{itemize}
\item for $p=3$, $\Out_{\FF'}(T)\geq \SL_2(\F_3)$ and $\Out_{\FF'}(V)\geq \SL_2(\F_3)$;
\item for $p>3$, $\Out_{\FF'}(T)\geq A_p$ and $\Out_{\FF'}(V)\geq \SL_2(\F_p)$.
\end{itemize}
To finish the proof, we make use of property (N3) of normal subsystems and the saturation of $\FF'$. We also distinguish the cases $p=3$ and $p\geq5$. Recall the outer automorphisms of $S$, $\phi$ and $\psi$, defined in Notation~\ref{not:phipsi}. Recall also that their restrictions to $V$ and $T$ are respectively given by $\phi|_V=\big(\begin{smallmatrix}\lambda & 0\\ 0 & 1 \end{smallmatrix} \big)$, $\psi|_V=\big(\begin{smallmatrix}1 & 0 \\ 0 & \lambda \end{smallmatrix} \big)$ (where $V = \gen{s, \zeta}$ and we use additive notation), and by $\phi|_{T}=\sigma$ ($\sigma\in\Sigma_p$ an element of order $(p-1)$ normalizing $s$) and $\psi|_{T}=\lambda\Id$, respectively.
\begin{itemize}
\item Consider first $p=3$. In this case, $-\Id\in\SL_2(\F_3) \leq \Aut_{\FF'}(V)$ must extend to $N_{S}(V)$ (see \cite[Definition 2.2 (II)]{BLO3}), which is strictly larger than $V$. Applying Alperin's fusion theorem \cite[Theorem 3.6]{BLO3}, this morphism must extend to an automorphism $\varphi\in\Aut_{\FF'}(S)$, since $S$, $T$ and $V$ are representatives of the only $\FF$-conjugacy classes of centric radical subgroups in $(S, \FF)$. As $\varphi\in\Aut_{\FF}(S)$, $\varphi=\phi^i\psi^j c_g$, with $g\in S$, and checking the restriction to $V$, we get $\varphi=\phi\psi c_g$. The restriction of $\varphi=\phi\psi c_g$ to the maximal torus $T$ will give an element in $\Aut_{\FF'}(T)$ with determinant $-1$ (see Notation \ref{not:phipsi} for details, as the restrictions specified there applied to the infinite case too). This implies that $\Aut_{\FF'}(T)=\Aut_{\FF}(T)$. Moreover, this argument can be applied now to $-\Id\in\SL_2(\F_3) \leq \Aut_{\FF'}(T_3)$, which must extend to $\psi c_g$, and restricting to $\Aut_{\FF'}(V)$ we get an element of determinant $-1$, obtaining that $\Aut_{\FF'}(V)=\Aut_{\FF}(V)$. Finally, this shows that $\Aut_{\FF'}(S)$ contains $\langle\phi\psi, \psi\rangle=\langle \phi, \psi\rangle$. 
\item Consider now the case $p\geq 5$: the part of the proof concerning the extension of elements in $\SL_2(\F_p)\leq \Aut_{\FF'}(V)$ is the same, replacing $-\Id$ by diagonal matrices of the type $(\begin{smallmatrix}\lambda & 0 \\ 0 & \lambda^{-1}\end{smallmatrix})$, for $\lambda \in \F_p^\times$, which will extend to $(\phi\psi^{-1})^\lambda \in \Aut_{\FF'}(S)$. The extension and restriction argument implies that $\Aut_{\FF'}(T)\geq A_p\rtimes C_{p-1} \cong \Aut_{\FF}(T)$. Finally, consider $\varphi$ an element of order $(p-1)/2$ in $N_{A_p}(\langle s \rangle)$ (we can assume $\varphi=\phi^2\in\Aut_\FF(S)$) and the induced action on $T$. This action must extend to $N_{S}(T)=S$ and the restriction to $V$ gives matrices of determinant $\lambda^2$, a square in $\F_p^\times$, obtaining $\Aut_{\FF'}(T)=\Aut_{\FF}(T)$. In this case, this also implies that $\Aut_{\FF'}(S)$ contains $\langle\phi\psi^{-1}, \phi^2\rangle$. 
\end{itemize}
Now the result follows, as we have seen that $\FF'$ must contain all the generators of $\FF$.

To see (b), we have to proceed similarly, but all the computations have been already done: any normal saturated fusion subsystem of $(S,\widetilde\FF,\widetilde\LL)$ must be over $S$ (as there is not any proper strongly $\widetilde\FF$-closed subgroup by Lemma~\ref{no-proper-nontrivial-strongly-closed-subgroups}) and must contain all the automorphims of order $p$. So, applying the computations in the poof of (a), we see that $\FF = O^{p'}(\widetilde\FF)$ (see Remark~\ref{rem_minnor}), and by Corollary~\ref{minnor} it follows that $\FF$ is a normal subsystem of index prime to $p$ of $\widetilde{\FF}$. These computations also show that $\Gamma_{p'}(\widetilde\FF)\cong\Z/2$, where $\Gamma_{p'}(\widetilde\FF)$ is the group considered in Theorem~\ref{app4}. This tells that all the possible saturated fusion subsystems of index prime to $p$ are in bijective correspondence with the subgroups of $\Z/2$, and this proved (b).

Finally, we prove part (c). Let $(S, \EE, \TT)$ be any of the $p$-local compact groups described in Theorem~\ref{FFpissaturated}. By the Hyperfocal Subgroup Theorem \cite[Theorem B.5]{Gonza2}, we have
\[
\pi_1(|\TT|^{\wedge}_p) = S/O^p_{\EE}(S),
\]
where $O^p_{\EE}(S)$ is the hyperfocal subgroup defined in Definition~\ref{defifoca}. By \cite[B.12-B.13]{Gonza2}, it follows that $O^p_{\EE}(S) \lneqq S$ is a strongly $\EE$-closed subgroup containing $T$, and hence we have $O^p_{\EE}(S) = S$ by Lemma~\ref{no-proper-nontrivial-strongly-closed-subgroups}.
\end{proof}

\subsection{On the exoticness of \texorpdfstring{$p$}{p}-local compact groups}

In this subsection we prove that there does not exist any compact Lie group or $p$-compact group  realizing any of the fusion systems described in Theorem~\ref{FFpissaturated}. We start proving that there does not exist any $p$-compact group with this fusion. In particular, as a connected compact Lie group corresponds to a $p$-compact group, this also shows that there does not exist any connected compact Lie group realizing these fusion systems.

Let us fix first the usual definitions and notations when working with $p$-compact groups (we refer to W. Dwyer and C. Wilkerson papers \cite{DW0} and \cite{DW2} for more details): a $p$-compact group is a triple $(X,BX,e)$ where $X$ is a space such that $H^*(X;\F_p)$ is finite, $BX$ a pointed $p$-complete space and $e\colon X \to \Omega(BX)$ is a homotopy equivalence. We refer to $X$ as a $p$-compact group and $BX$ and $e$ are assumed. If $X$ and $Y$ are $p$-compact groups, a homomorphism $f\colon X \to Y$ is a pointed map $Bf \colon BX \to BY$. Two homomorphisms $f,f'\colon X \to Y$ are conjugate if $Bf$ and $Bf'$ are freely homotopic.

The following is a general result about $p$-compact groups and $p$-local compact groups. In order to avoid confusion with the notation in Theorem \ref{FFpissaturated}, let $(R, \EE, \TT)$ be a $p$-local compact group, and let $P \leq R$ be a fully $\EE$-centralized subgroup. We also fix the notation $\theta\colon BR\to |\TT|^\wedge_p$ as the composition of the inclusion of the Sylow $p$-subgroup $BR\to |\TT|$ \cite[Proposition 4.4]{BLO3} and $p$-completion.  In this situation, there is a well-defined notion of centralizer $p$-local compact group of $P$, denoted by $(C_R(P), C_{\EE}(P), C_{\TT}(P))$, see \cite[Section 2]{BLO6} for the explicit definition and properties. The following proposition describes the relation between algebraic centralizers and mapping spaces for $p$-local compact groups associated to $p$-compact groups.

\begin{proposition}\label{prop_p_compact}
Let $(R,\EE,\TT)$ a $p$-local compact group such that $|\TT|^\wedge_p\simeq BX$, where $(X,BX,e)$ is a $p$-compact group. Then, the following holds.
\begin{enumerate}
	\item $X$ is connected if and only if all the elements in $R$ are $\EE$-conjugate to elements in the maximal torus.
	\item Let $P$ be a fully $\EE$-centralized subgroup of $R$. Let also $(C_R(P), C_{\EE}(P), C_{\TT}(P))$ be the centralizer $p$-local compact group of $P$, and let $\theta|_{BP}\colon BP \to |\TT|^\wedge_p$ be the composition of the inclusion $BP \to BR$ with $\theta\colon BR \to |\TT|^\wedge_p$. Then there is a homotopy equivalence $\Map(BP,|\TT|^\wedge_p)_{\theta|_{BP}}\simeq |C_\TT(P)|^\wedge_p$.
\end{enumerate}
\end{proposition}
\begin{proof}
By \cite[Proposition 10.1 and Theorem 6.3(a)]{BLO3}, we can consider $f\colon R\to X$ a maximal discrete $p$-toral subgroup and we use this notation, and the corresponding $Bf\colon BR \to BX$ in all this proof. Moreover, as we are also considering the $p$-local compact group structure of $(X,BX,e)$, which is $(R,\EE,\TT)$, we assume that the composition of $\theta\colon BR \to |\TT|^\wedge_p$ with the fixed homotopy equivalence $|\TT|^\wedge_p\simeq BX$ is $Bf\colon BR\to BX$: we can assume this, as all Sylow $p$-subgroups in a $p$-compact group are conjugated \cite[Proposition 10.1]{BLO3}.
	
We prove first (a): Consider $T \subset R$ the inclusion of the maximal torus in $R$. Define $i\colon T \to X$ to be the composition of the inclusion map and $f$. Recall from 
\cite[Definition 10.2]{BLO3} that the saturated fusion system over $R$ corresponding to $X$, which is denoted by $\EE_{R,f}(X)$, is defined as:
	\[
	\Mor_{\EE_{R,f}(X)}(P,Q)\definicio\{\varphi \in \Hom(P,Q) \mid Bf|_{BQ}\circ B\varphi \simeq Bf|_{BP} \} .
	\] 
	Assume first that $X$ is connected and let $x\in R$. The composition of the inclusion of $\langle x \rangle$ in $R$ and $f$ gives a monomorphism $g\colon \Z/p^n \to X$, where $p^n=|\langle x \rangle|$.
	By \cite[Proposition 3.11]{DW2}, as $X$ is connected,  any morphism $g\colon \Z/p^n \to X$ extends to $\overline{g}\colon \Z/p^\infty \to X$. Applying now \cite[Proposition 8.11]{DW0} we get that there is $h\colon \Z/p^\infty \to T$ such that $i\cdot h$ is conjugate to $\overline g$. The restriction of $h$ to $\Z/p^n$ gives a morphism $\varphi \in \Mor_{\EE_{R,f}(X)}(\langle x \rangle,T)$.\\
	If $X$ is not connected, also by \cite[Proposition 3.11]{DW2}, there exists $g\colon \Z/p^n\to X$ which does not extend to $\Z/p^{\infty}$. By the maximality of $R$, this map factors through $\widetilde{g}\colon\Z/p^n\to R$. Consider $x\definicio\widetilde g(1)\in R$. This element cannot be conjugated to the maximal torus, otherwise, we would be able to extend $\widetilde g$ to a map from $\Z/p^\infty$, providing an extension of $g$.
	
    To prove (b) consider first $(C_R(P), C_\EE(P), C_\TT(P))$, the $p$-local compact group defined as the centralizer of $P$ in $(R,\EE,\TT)$. This $p$-local compact group exists by \cite[Theorem 2.3]{BLO6} because $P$ is fully $\EE$-centralized. 
    This way, we may consider the mapping space $\Map(BP,|\TT|^\wedge_p)_{\theta|_{BP}}$ as a $p$-compact group. By \cite[Section 10]{BLO3}, this has a $p$-local compact group structure which we denote as $(R',\EE',\TT')$, with $|\TT'|^\wedge_p = \Map(BP,|\TT|^\wedge_p)_{\theta|_{BP}}$.
    
    By \cite[Proposition 10.4]{BLO3}, as $P$ is fully $\EE$-centralized, the map $\gamma_P\colon BC_R(P)\to \Map(BP,|\TT|^\wedge_p)_{\theta|_{BP}}=|\TT'|^\wedge_p$ defined as the adjoint to the composite
    \begin{equation}\label{adjoint}
    B(P\times C_R(P)) \stackrel{B\mu}{\longrightarrow}BR\stackrel{\theta}{\longrightarrow} |\TT|^\wedge_p
    \end{equation}
    with $\mu \colon P \times C_R(P)\to R$ the multiplication, is a Sylow $p$-subgroup of $|\TT'|^\wedge_p$. So we get that we can consider $R'=C_R(P)$ and $\gamma_P\colon BR' \to |\TT'|^\wedge_p$ a Sylow map.
     
	Thus, it remains to prove that for all $Q$, $Q'$ subgroups of $C_R(P)$, the homomorphism $\varphi\colon Q \to Q'$ belongs to $\Hom_{C_\EE(P)}(Q,Q')$ if and only if $\varphi$ belongs to $\Hom_{\EE'}(Q,Q')$. 

	Consider now the following three diagrams:
	\[
	\xymatrix{BP\times BQ \ar[r]^{\Id\times B\varphi} \ar[d]_{B\mu|_{P\times Q}} \ar @{} [dr] | {(*)} & BP\times BQ' \ar[d]^{B\mu|_{P\times Q'}} \\
		BPQ \ar[r]_{B\widetilde{\varphi}} & BPQ'} 
	\xymatrix{BPQ \ar[dr]_{\theta|_{BPQ}}  \ar[rr]^{B\widetilde{\varphi}} \ar @{} [drr] | {(**)} & & BPQ'\ar[dl]^{\theta|_{BPQ'}} \\
		& |\TT|^\wedge_p & }
	\]
	\[
	\xymatrix{BP\times BQ \ar[rr]^{\Id\times B\varphi} \ar[d]_{B\mu|_{P\times Q}} \ar @{} [ddrr] | {(***)} & & BP\times BQ' \ar[d]^{B\mu|_{P\times Q'}} \\
		BPQ \ar[dr]_{\theta|_{BPQ}} & & BPQ'\ar[dl]^{\theta|_{BPQ'}} \\
		& |\TT|^\wedge_p &} 
	\]	

	And split the proof in several steps:
	
	\textbf{Step 1:} A group homomorphism $\varphi\colon Q \to Q'$ belongs to $\Hom_{C_\EE(P)}(Q,Q')$ if and only if we can construct the homotopy commutative diagram $(**)$:
	
	If $\varphi\colon Q \to Q'$ belongs to $\Hom_{C_\EE(P)}(Q,Q')$, there exists $\widetilde{\varphi}\in\Hom_\EE(PQ,PQ)$ such that $\widetilde{\varphi}|_{Q}=\varphi$ and $\widetilde{\varphi}|_P=\Id_P$. But, by definition of $\EE$ as a fusion system corresponding to a $p$-compact group $X$ and the fixed notation at the beginning of this proof, this is equivalent to $\theta|_{BPQ} \simeq \theta|_{BPQ'} \circ B\widetilde{\varphi}$.
     
    \textbf{Step 2:} A group homomorphism $\varphi\colon Q \to Q'$ belongs to $\Hom_{\EE'}(Q,Q')$ if and only if we can construct the homotopy commutative diagram $(***)$:
    
    Recall the inclusion of the Sylow $p$-subgroups $\gamma_P\colon BC_R(P) \to \Map(BP,|\TT|^\wedge_p)_{\theta|_{BP}}$ considered above. Now $\varphi\in\Hom_{\EE'}(Q,Q')$ if and only if $B\gamma_P|_{BQ} \simeq B\gamma_P|_{BQ'}\circ B\varphi$. And, considering adjoint maps (see Equation~\eqref{adjoint}), this is equivalent to verifying that the composition $\theta|_{BPQ}\circ B\mu|_{P\times Q}$ and  $\theta_{BPQ'} \circ B\mu|_{P\times Q'} \circ \Id \times B\varphi$ are homotopy equivalent, obtaining diagram $(***)$.

	\textbf{Step 3:} Homotopy commutative diagram $(*)$ can be constructed for any $\varphi \in \Hom_{C_\EE(P)}(Q,Q')$ and any $\varphi \in \Hom_{\EE'}(Q,Q')$:

	If $\varphi \in \Hom_{C_\EE(P)}(Q,Q')$, we can construct the commutative diagram $(*)$ by definition of $C_\EE(P)$.

	Assume now $\varphi \in \Hom_{\EE'}(Q,Q')$ (so, we have commutative diagram $(***)$). Consider diagram $(*)$ at the level of groups:
	\[ \xymatrix{P\times Q \ar[r]^{\Id\times \varphi} \ar[d]_{\mu} & P\times Q' \ar[d]^{\mu} \\
		PQ \ar[r]_{\widetilde{\varphi}} & PQ'}
	\] 
	Which exists (in the category of groups) if all the elements of $P$ commute with all the elements of $Q$ and $Q'$, and $\varphi|_{P\cap Q}=\Id|_{P\cap Q}$ (here we use $\Id$ as notation for the corresponding inclusion). Moreover, in this case, $\widetilde{\varphi} \colon P \times Q \to PQ$ is uniquely defined as $\widetilde{\varphi}(ab) = a\varphi(b)$.	As $Q,Q'\leq C_R(P)$, the commutation condition is satisfied, so we have to check that $\varphi|_{P\cap Q}=\Id|_{P\cap Q}$: $P \cap Q \leq Z(P)$ and the homomorphism $B\varphi|_{B(P \cap Q)}$ composed with the inclusion of $BC_R(P)$ in $|\TT'|^\wedge_p$ is central (as $p$-compact groups), so, by \cite[Lemma~6.5]{DW2}, the morphism of $p$-compact toral groups $B\varphi|_{B(P \cap Q)}\colon B(P \cap Q)\to BPQ'$ is unique. By \cite[Proposition 1.10]{BLO3}, $B\varphi|_{B(P \cap Q)}$ corresponds to a group morphism $\varphi'\colon P\cap Q \to PQ'$ which is the composition of the inclusion (because of the commutative diagram $(***)$) with a conjugation in $PQ'$. But, as $PQ'$ centralizes $Z(P)$, conjugation by any element in $PQ'$  is the identity in $Z(P)$ and, in particular, in $P\cap Q\leq Z(P)$. So $\varphi|_{P \cap Q} = \Id|_{P \cap Q}$.

    \textbf{Step 4:} As diagram $(***)$ can be constructed from $(*)$ and $(**)$, we get the inclusion $\Hom_{C_\EE(P)}(Q,Q') \subset \Hom_{\EE'}(Q,Q')$.

    \textbf{Step 5:} If $\varphi$ is a morphism in $\EE'$, then diagram $(***)$ is homotopy commutative. We want to see that this implies that $(**)$ is also homotopy commutative. For that, consider $K$ to be the kernel of $\mu|_{P\times Q}$. The map from $BK\to |\TT|^\wedge_p$ is the composition $\theta|_{BPQ}\circ B\mu|_{BK}$, and it is a central map in $|\TT|^\wedge_p$, so $\Map(BK,|\TT|^\wedge_p)_{\theta|_{BPQ}\circ B\mu|_{BK}}\simeq |\TT|^\wedge_p$ (here we are using that $|\TT|^\wedge_p$ is the classifying space of a $p$-compact group). This allows us to see that the map $|\TT|^\wedge_p \to \Map(BK,|\TT|^\wedge_p)_{\theta|_{BPQ}\circ B\mu|_{BK}}$ which sends each point $t$ to the constant map $t$ is an homotopy equivalence. So, we can apply Zabrodsky Lemma as stated in \cite[Proposition 3.5]{DwyerBCAT94} and we get that $B\mu|_{P\times Q}$ induce an equivalence $\Map(BPQ,|\TT|^\wedge_p)\to \Map(BP\times BQ,|\TT|^\wedge_p)_{[\theta|_{BPQ}\circ B\mu|_{P\times Q}]}$. The class $[\theta|_{BPQ}]\in\pi_0(\Map(BPQ,|\TT|^\wedge_p))$ corresponds to $[\theta|_{BPQ}\circ B\mu|_{P\times Q}]\in\pi_0(\Map(BP\times BQ,|\TT|^\wedge_p))$, and the class $[\theta|_{BPQ'}\circ B\widetilde{\varphi}]$ corresponds to $[\theta|_{BPQ'}\circ B\widetilde{\varphi}\circ B\mu|_{P\times Q}]$. Using diagrams $(*)$ and $(***)$, we have that the following classes in $\pi_0(\Map(BP\times BQ,|BK|^\wedge_p))$ are the same:
    \[
    [\theta|_{BPQ'}\circ B\widetilde{\varphi}\circ B\mu|_{P\times Q}]=
    [\theta|_{BPQ'}\circ B\mu|_{P\times Q'} \circ \Id\times B\varphi]=
    [\theta|_{BPQ}\circ B\mu|_{P\times Q}].
    \]
    This implies that $(**)$ is also homotopy commutative and $\varphi\in\Hom_{C_\EE(P)}(Q,Q')$, which finishes the proof.
\end{proof}

\begin{remark}
In \cite[Theorem D]{Gonza3} the first author proves a more general version of Proposition~\ref{prop_p_compact} (b). The proof we give above is independent from \cite{Gonza3}.
\end{remark}

\begin{theorem}\label{exoticaspcompact}
There does not exist any $p$-compact group realizing the $p$-local compact groups in Theorem \ref{FFpissaturated}.
\end{theorem}
Before the proof of the theorem we need a result which follows from a result by K.~Ishiguro \cite[Proposition 3.1]{ishiguro}.
\begin{lemma}\label{lemma_ishiguro}
Let $p$ be a prime number and $H$ a finite non $p$-nilpotent group acting on a torus $T$. Then, there does not exist any $p$-compact group realizing the fusion system of $T\rtimes H$ over the prime $p$.
\end{lemma}
\begin{proof}
Consider the compact Lie group $G\definicio T\rtimes H$. By \cite[Theorem 9.10]{BLO3} there is a $p$-local compact group $(R,\EE,\TT)$ with $\EE$ the fusion system of $G$ over a Sylow $p$-subgroup $S$ and $|\TT|^\wedge_p \simeq BG^\wedge_p$. Assume there is a $p$-compact group $X$ realizing also the p-local compact group $(R,\EE,\TT)$. Then, by \cite[Theorem 10.7]{BLO3}, $|\TT|^\wedge_p \simeq BX$, hence $BG^\wedge_p \simeq BX$. In this case, by \cite[Proposition 3.1]{ishiguro}, the group of components of $G$ must be a $p$-nilpotent group, in contradiction with the hypothesis in $H$.
\end{proof}

\begin{proof}[Proof of Theorem~\ref{exoticaspcompact}]
Let $(S, \EE, \TT)$ be any of the examples in Theorem~\ref{FFpissaturated}, and assume that there exists some $p$-compact group $X$ such that $BX \simeq |\LL|^\wedge_p$. Let $Z$ be the centre of $S$, which is isomorphic to $\Z/p\Z$. By definition, it is a fully centralized subgroup, so we can construct the centralizer of $Z$ in $(S,\EE,\TT)$, which is again a $p$-local compact group that we denote by
$(S,C_{\EE}(Z),C_{\TT}(Z))$. By Lemma~\ref{centralZ}, $C_{\EE}(Z)$ is the fusion system over $S$ of either $T \rtimes \Sigma_3$ (if $p = 3$), $T \rtimes A_p$ (if $p \geq 5$ and $\EE = \FF$), or $T \rtimes \Sigma_p$ (if $p \geq 5$ and $\EE = \widetilde{\FF}$). As neither $\Sigma_p$ for $p\geq 3$, nor $A_p$ for $p\geq5$ are $p$-nilpotent, it follows by Lemma~\ref{lemma_ishiguro} that none of these is the fusion system of a $p$-compact group.

If we denote by $C_X(Z)$ the centralizer in $X$ of the composition of maps $Z \hookrightarrow S \to X$ we have that $C_X(Z)$ is again a $p$-compact group by \cite{DW0}. But by Proposition~\ref{prop_p_compact} (b), $C_X(Z)\simeq |C_{\TT}(Z)|^\wedge_p$, so $(S,C_{\EE}(Z))$ is the fusion system of a $p$-compact group, getting a contradiction with the previous paragraph.
\end{proof}

Until now we have proved that the $p$-local compact groups described in Theorem~\ref{FFpissaturated} cannot be realized by $p$-compact groups. This result includes the impossibility of these $p$-local compact groups to be realized by compact Lie groups whose group of components is a $p$-group. In order to prove that the $p$-local compact groups of Theorem~\ref{FFpissaturated} are not realized by any compact Lie group,  it remains to eliminate the case of compact Lie groups whose group of components is not a $p$-group.

\begin{theorem}\label{exoticasLiegroup}
There does not exist any compact Lie group realizing the $p$-local compact groups of Theorem~\ref{FFpissaturated}.
\end{theorem}

\begin{proof}
Consider first the fusion system $(S, \FF)$, for $p \geq 3$, in Theorem~\ref{FFpissaturated}, and assume that there is a compact Lie group $G$ such that $\FF \cong \FF_{S}(G)$ for $S \in \Syl_p(G)$. Let $G_0 \trianglelefteq G$ be the 
connected component of the identity in $G$. By Lemma~\ref{normal_subgroup_normal_subsystem}, we have that $T \leq S \cap G_0$ is strongly closed in 
$\FF_{S}(G)$, but, by Lemma~\ref{no-proper-nontrivial-strongly-closed-subgroups}, $\FF$ has no proper nontrivial strongly closed subgroups, hence 
$S \leq G_0$. Then, again by Lemma~\ref{normal_subgroup_normal_subsystem}, $\FF_{S}(G_0) \trianglelefteq \FF_{S}(G)$, but, since $\FF$ is a 
simple saturated fusion system by Proposition~\ref{prop_FFp_simple}(a), we must have $\FF_{S}(G_0) \cong \FF_{S}(G)$. This is impossible since a 
connected compact Lie group gives rise to a $p$-compact group, and the saturated fusion system $\FF$ is not realized by any $p$-compact group by 
Theorem~\ref{exoticaspcompact}. 

Let now $p \geq 5$, consider the $p$-local compact group $(S, \widetilde{\FF}, \widetilde{\LL})$ in Theorem~\ref{FFpissaturated}, and assume that there is a compact Lie group $\widetilde{G}$ with $S \in \Syl_p(\widetilde{G})$ and such that $(S, \widetilde{\FF}, \widetilde{\LL}) \cong (S, \FF_{S}(\widetilde{G}), \LL_{S}^c(\widetilde{G}))$. Let again $\widetilde{G}_0 \trianglelefteq \widetilde{G}$ be 
the connected component of the identity in $\widetilde{G}$. As before, by Lemma~\ref{normal_subgroup_normal_subsystem}, we have $S \leq 
\widetilde{G}_0$ and $\FF_{S}(\widetilde{G}_0) \trianglelefteq \FF_{S}(\widetilde{G})$. Then, we know by Proposition~\ref{prop_FFp_simple}(b) that $\FF$ is the only proper nontrivial normal subsystem of $(S,\widetilde\FF,\widetilde\LL)$. Therefore, in this case we must have 
$\FF_{S}(\widetilde{G}_0) \cong \FF$ or $\FF_{S}(\widetilde{G}_0) \cong \widetilde\FF$, but we have proved in Theorem~\ref{exoticaspcompact} 
that there is no $p$-compact group realizing any of these two fusion systems, hence $\widetilde{G}$ cannot exist.
\end{proof}


\appendix

\section{Fusion subsystems of index prime to \texorpdfstring{$p$}{p} in \texorpdfstring{$p$}{p}-local compact groups}\label{appendixA}

In this section we generalize to the compact case the results in \cite{BCGLO2} about detection of subsystems of index prime to $p$ of a given fusion system (see Definition~\ref{defifoca}). We also show that the \textit{minimal} subsystem of index prime to $p$ is always a normal subsystem. Throughout this appendix, we fix a $p$-local compact group $\g = (S, \FF, \LL)$.

\begin{definition}
A subgroup $P \leq S$ is \textit{$\FF$-quasicentric} if, for all $Q \in P^{\FF}$ that is fully $\FF$-centralized, the centralizer fusion system $C_{\FF}(Q)$ is the fusion system of $C_S(Q)$.
\end{definition}

We shall also use the following notation.
\begin{itemize}
\item For a subset $\hh \subseteq \Ob(\FF)$, $\FF_{\hh} \subseteq \FF$ denotes the full subcategory of $\FF$ with object set $\hh$. The set of all morphisms in $\FF_{\hh}$ is denoted $\Mor(\FF_{\hh})$. In the particular case where $\hh$ is the set of all $\FF$-quasicentric subgroups of $S$, we simply write $\FF^q$ instead of $\FF_{\hh}$.
\item For a (discrete) group $\Gamma$, $\mathfrak{Sub}(\Gamma)$ denotes the set of nonempty subsets of $\Gamma$.
\end{itemize}
The main tool to detect subsystems of a given fusion system are the so-called \emph{fusion mapping triples}, which were already generalized from \cite{BCGLO2} to the context of $p$-local compact groups in \cite[Definition B.7]{Gonza2}.

When constructing fusion mapping triples we may have to deal with infinitely many conjugacy classes of subgroups of $S$. The \textit{bullet functor} $\functor \colon \FF \to \FF$ defined in \cite[Section 3]{BLO3} is the tool to reduce to situations involving only finitely many $\FF$-conjugacy classes. We refer to \cite{BLO3} for the properties of $(-)^{\bullet}$ which we use in this appendix. Given a nonempty full subcategory $\FF_0 \subseteq \FF$, we denote by $\FF_0^{\bullet} \subseteq \FF_0$ the full subcategory whose objects are the subgroups $P \in \Ob(\FF_0)$ such that $P = P^{\bullet}$. A priori, $\FF_0^{\bullet}$ could be empty, but this is not be the case when $\FF_0$ is closed by over-groups, as $P \leq P^{\bullet}$ for all $P \leq S$. The next result constitutes the key to inductively construct fusion mapping triples.

\begin{lemma}\label{fmt2}
Let $\hh_0 \subseteq \Ob(\FF^{\bullet q})$ be a nonempty subset closed by $\FF$-conjugacy and over-groups (in $\FF^{\bullet}$) and $\pp$ be an $\FF$-conjugacy class in $\FF^{\bullet q}$ maximal among those not contained in $\hh_0$. Set $\hh = \hh_0 \bigcup \pp$ and let $\FF_{\hh_0} \subseteq \FF_{\hh} \subseteq \FF^{\bullet q}$ be the corresponding full subcategories. Finally, let $(\Gamma, \theta, \Theta)$ be a fusion mapping triple for $\FF_{\hh_0}$ and, for each $P \in \pp$ which is fully $\FF$-normalized, fix a homomorphism
\[
\Theta_P \colon \Aut_{\FF}(P) \longrightarrow N_{\Gamma}(\theta(C_S(P)))/\theta(C_S(P))
\]
satisfying the following conditions:
\begin{enumerate}
\item $x \cdot \theta(f) \cdot x^{-1} = \theta(f(g))$ for all $g \in P$, $f \in \Aut_{\FF}(P)$ and $x \in \Theta_P(f)$; and
\item $\Theta_P(f) \supseteq \Theta(f')$ for all $P \lneqq Q \leq S$ such that $P \lhd Q$ and $Q$ is fully $\FF$-normalized, and for all $f \in \Aut_{\FF}(P)$ and $f' \in \Aut_{\FF}(Q)$ such that $f = f'|_P$.
\end{enumerate}
Then, there exists a unique extension of $\Theta$ to a fusion mapping triple $(\Gamma, \theta, \widetilde{\Theta})$ on $\FF_{\hh}$ such that $\widetilde{\Theta}(f) = \Theta_P(f)$ for all $f \in \Aut_{\FF}(P)$.
\end{lemma}
\begin{proof}
This is \cite[Lemma B.9]{Gonza2} with minor modifications to restrict to $\FF^{\bullet q}$.
\end{proof}

\begin{lemma}\label{app51}
Let $Q \leq S$ be an $\FF$-quasicentric subgroup, let $P \leq S$ be a subgroup such that $Q \lhd P$, and let $\varphi, \varphi' \in \Hom_{\FF}(P,S)$ be such that $\varphi|_Q = \varphi'|_Q$, and such that $\varphi(Q)$ is fully $\FF$-centralized. Then there exists some $x \in C_S(\varphi(Q))$ such that $\varphi' = c_x \circ \varphi$.
\end{lemma}

\begin{proof}
As a first simplification, we may replace $P$ by $\varphi(P)$ and $Q$ by $\varphi(Q)$. This way, we may assume that $\varphi' = \incl_P^S$ and $\varphi|_Q = \Id_Q$. Next, we justify reducing to the case where $Q = Q^{\bullet}$. Indeed, by \cite[Lemma 1.23]{Gonza3}, $Q$ is fully $\FF$-centralized if and only if $Q^{\bullet}$ is so, in which case we have $C_{\FF}(Q) = C_{\FF}(Q^{\bullet})$. In particular, this implies that $Q^{\bullet}$ is also $\FF$-quasicentric. The properties of the functor $(-)^{\bullet}$ then justify the restriction to the case $Q = Q^{\bullet}$. This last reduction allows us to do induction on $|Q|$ within the set of objects of $\FF^{\bullet}$, as this category contains finitely many $S$-conjugacy classes of objects. Finally, note that the statement is true if $Q$ is $\FF$-centric by \cite[Proposition 2.8]{BLO3}.

As $Q$ is $\FF$-quasicentric and fully $\FF$-centralized, it follows that $\varphi|_{C_P(Q)}$ corresponds to the conjugation homomorphism induced by some $x \in C_S(Q)$. Thus, after composing with $(c_x)^{-1}$, we may assume without loss of generality that $\varphi|_{C_P(Q)Q} = \Id$. Moreover, if $C_P(Q)Q \geq Q$, then the statement follows by induction on $|Q|$, simply by taking $\overline{P} = P^{\bullet}$ and $\overline{Q} = (C_P(Q)Q)^{\bullet}$.

Assume then that $C_P(Q) \leq Q$, and set $K = \Aut_P(Q)$. Following the notation of \cite[Section 2]{BLO6}, we write
\[
N_S^K(Q) = \{x \in N_S(Q) \mid c_x \in K\}.
\]
and let $N_{\FF}^K(Q)$ be the fusion system over $N_S^K(Q)$ with morphism sets
\[
\Hom_{N_{\FF}^K(Q)}(R, R') = \{\gamma \in \Hom_{\FF}(R, R') \mid \exists \alpha \in \Hom_{\FF}(RQ, R'Q) \colon \alpha|_Q \in K \mbox{ and } \alpha|_R = \gamma \}.
\]
Note that $P$, $\varphi(P)$ and $C_S(Q)$ are subgroups of $N_S^K(Q)$. By \cite[Lemma 2.2]{BLO6}, if $Q$ is not fully $K$-normalized in $\FF$, then there exists some $\lambda \in \Hom_{\FF}(N_S^K(Q), S)$ such that $\lambda(Q)$ is fully $\lambda K \lambda^{-1}$-normalized in $\FF$. Hence, by replacing all the subgroups by their images through $\lambda$ if necessary, we may assume that $Q$ is already fully $K$-normalized in $\FF$.

The fusion system $N_{\FF}^K(Q)$ is saturated by \cite[Theorem 2.3]{BLO6}. Replacing $\FF$ by $N_{\FF}^K(Q)$ if needed, we can then assume that $S = N_S^K(Q) = P C_S(Q)$, and $\FF = N_{\FF}^K(Q)$. In particular, each morphism in $\FF$ extends to a morphism whose domain contains $Q$ and whose restriction to $Q$ corresponds to conjugation by some element of $P$.

Fix $\alpha \in \Hom_{\FF}(P, S)$ such that $\alpha(P)$ is fully normalized in $\FF$. Since $\alpha|_Q = c_g$ for some $g \in P$, we can replace $\alpha$ by $\alpha \circ (c_g)^{-1}$, so that $\alpha|_Q = \Id$. If $\alpha$ and $\alpha \circ \varphi^{-1}$ are both given by conjugation by elements in $C_S(Q)$, then so is $\varphi$. Thus, to prove the statement it is enough to prove the same statement under the extra assumption that $\varphi(P)$ is fully normalized in $\FF$.

Next, note that $(C_S(Q)Q)/Q$ is a nontrivial normal subgroup of $N_S(Q)/Q = S/Q$. Hence there is some $x \in C_S(Q) \setminus Q$ such that $1 \neq xQ \in Z(S/Q)$. It follows that $x \in N_S(P)$ and $x$ acts (via conjugation) as the identity on both $Q$ and $P/Q$. Thus,
\[
c_x \in \Ker(\Aut_{\FF}(P) \to \Aut_{\FF}(Q) \times \Aut(P/Q)),
\]
which is a normal $p$-subgroup of $\Aut_{\FF}(P)$ by \cite[Lemma 1.7]{Gonza3}. In addition, $\Aut_S(\varphi(P)) \in \Syl_p(\Aut_{\FF}(\varphi(P)))$ as $\varphi(P)$ is fully normalized in $\FF$. It follows that $\varphi c_x \varphi^{-1} \in \Aut_S(\varphi(P))$ (after replacing $\varphi$ by $\varphi \circ \omega$ for certain $\omega \in \Aut_{\FF}(\varphi(P))$ is needed). Thus, $x \in N_{\varphi}$, and $Q \lneqq N_{\varphi}$. By axiom (II) of saturated fusion systems, $\varphi$ extends to some $\overline{\varphi} \in \Hom_{\FF}(N_{\varphi}, S)$.

Set $R = (N_{\varphi})^{\bullet}$. Set also $\overline{Q} = (C_R(Q)Q)^{\bullet}$ and $\overline{P} = N_R(\overline{Q})$. Note that $\overline{Q} \lhd \overline{P}$. By construction, $x \in \overline{Q} \setminus Q$. As $Q$ is $\FF$-quasicentric, $\overline{\varphi}|_{C_{\overline{P}}(Q)}$ corresponds to conjugation by some element $g \in C_S(Q)$, and we can replace $\overline{\varphi}$ by $\overline{\varphi} \circ (c_g)^{-1}$, so that $\overline{\varphi}|_{\overline{Q}} = \Id$. As $\overline{Q} \gneqq Q$, this finishes the induction step.
\end{proof}

\begin{lemma}\label{app5}
Let $(\Gamma, \theta, \Theta)$ be a fusion mapping triple on $\FF^c$. Then there is a unique extension
\[
\widetilde{\Theta} \colon \Mor(\FF^q) \longrightarrow \mathfrak{Sub}(\Gamma)
\]
of $\Theta$ such that $(\Gamma, \theta, \widetilde{\Theta})$ is a fusion mapping triple on $\FF^q$.
\end{lemma}
\begin{proof}
By the properties of the functor $(-)^{\bullet}$, we can restrict the fusion mapping triple $(\Gamma, \theta, \Theta)$ to a fusion mapping triple $(\Gamma, \theta, \Theta^{\bullet})$ on $\FF^{\bullet c}$. Indeed, simply define $\Theta^{\bullet} = \Theta \circ \incl \colon \Mor(\FF_0^{\bullet}) \to \mathfrak{Sub}(\Gamma)$. Once we extend this fusion mapping triple to all of $\FF^{\bullet q}$ we can then extend it to $\FF^q$ using again the properties of the functor $(-)^{\bullet}$. Since there is no place for confusion we denote $\Theta^{\bullet}$ simply by $\Theta$.

Let $\hh_0 \subseteq \Ob(\FF^{\bullet q})$ be a set closed under $\FF$-conjugacy and over-groups (in $\FF^{\bullet q}$), and such that it contains $\Ob(\FF^{\bullet c})$, and let $\pp$ be a conjugacy class in $\FF^{\bullet q}$, maximal among those not in $\hh_0$. We want to extend $\Theta$ to $\hh = \hh_0 \cup \pp$.

Let $P \in \pp$ be fully $\FF$-normalized. For each $\alpha\in\Aut_{\FF}(P)$, there is an extension $\beta \in \Aut_{\FF}(R)$, where $R = P \cdot C_S(P)$, which in turn induces a unique $\beta^{\bullet} \in \Aut_{\FF}(R^{\bullet})$. Furthermore, by \cite[Proposition 2.7]{BLO3} both $R$ and $R^{\bullet}$ are $\FF$-centric (because $P$ is fully $\FF$-normalized), and in particular $R^{\bullet} \in \hh_0$. Thus we can define a map
\[
\Theta_P \colon \Aut_{\FF}(P) \longrightarrow \mathfrak{Sub}(N_{\Gamma}(\theta(C_S(P))))
\]
by the formula $\Theta_P(\alpha) = \Theta(\beta^{\bullet}) \cdot \theta(C_S(P))$. By properties (i) and (ii) of fusion mapping triples, $\Theta(\beta^{\bullet})$ is a left coset of $\theta(C_S(R))$ (because $Z(R) = Z(R^{\bullet})$ by \cite[Lemma 3.2 (d)]{BLO3}), and by (iv) it is also a right coset (where the left and right coset representatives can be chosen to be the same). Hence $\Theta_P(\alpha)$ is a left and right coset of $\theta(C_S(P))$.

If $\beta' \in \Aut_{\FF}(P)$ is any other extension of $\alpha$, then by Lemma~\ref{app51} there is some $g \in C_S(P)$ such that $\beta' = c_g \circ \beta$, and then $\Theta((\beta')^{\bullet}) = \Theta(c_g \, \beta^{\bullet}) = \theta(g) \Theta(\beta^{\bullet})$, and
\[
\begin{aligned}
\Theta((\beta')^{\bullet}) \cdot \theta(C_S(P)) & = \theta(g) \cdot \Theta(\beta^{\bullet}) \cdot \theta(C_S(P)) = \\
 & = \Theta(\beta^{\bullet}) \cdot \theta(\beta^{\bullet}(g)) \cdot \theta(C_S(P)) = \Theta(\beta^{\bullet}) \cdot \theta(C_S(P))
\end{aligned}
\]
and so the definition of $\Theta_P(\alpha)$ is independent of the choice of the extension of $\alpha$. This shows that $\Theta_P$ is well defined.

Note also that $\Theta_P$ respects compositions and, since $\Theta_P(\alpha) = x \cdot \theta(C_S(P)) = \theta(C_S(P)) \cdot x$ for some $x \in \Gamma$, we conclude that $x \in N_{\Gamma}(\theta(C_S(P))$. Thus $\Theta_P$ induces a homomorphism
\[
\Theta_P \colon \Aut_{\FF}(P) \longrightarrow N_{\Gamma}(\theta(C_S(P)))/\theta(C_S(P)).
\]
We can now apply Lemma~\ref{fmt2} to extend $\Theta$ to $\hh$.

If $\alpha \in \Aut_{\FF}(P)$ and $x \in \Theta_P(\alpha)$, then $x = y \cdot \theta(h)$ for some $h \in C_S(P)$ and $y \in \Theta(\beta^{\bullet})$, where $\beta^{\bullet}$ is some extension of $\alpha$ to $R = P \cdot C_S(P)$. Hence, for any $g \in P$,
\[
x  \theta(g)  x^{-1} = y \theta(hgh^{-1}) y^{-1} = y \theta(g) y^{-1} = \theta(\beta^{\bullet}(g)) = \theta(\alpha(g)).
\]
This shows that condition (i) in Lemma~\ref{fmt2} holds.

Assume now that $P \lneqq Q \leq N_S(P)$, and let $\alpha \in \Aut_{\FF}(P)$, $\beta \in \Aut_{\FF}(Q)$ be such that $\alpha = \beta|_P$. Then, in the notation of axiom (II) for saturated fusion systems, $Q \cdot C_S(P) \leq N_{\alpha}$, and hence $\alpha$ extends to some other $\gamma \in \Aut_{\FF}(Q \cdot C_S(P))$, and
\[
\Theta_P(\alpha) = \Theta(\gamma^{\bullet}) \cdot \theta(C_S(P))
\]
by definition of $\Theta_P$. By Lemma~\ref{app51}, $\gamma|_Q = c_g \circ \beta$ for some $g \in C_S(P)$, and hence by definition of fusion mapping triple, $\Theta(\gamma^{\bullet}) = \Theta(c_g \circ \beta^{\bullet}) = \theta(g) \cdot \Theta(\beta^{\bullet})$, and
\[
\begin{aligned}
\Theta_P(\alpha) & = \theta(g) \cdot \Theta(\beta^{\bullet}) \cdot \theta(C_S(P)) = \\
 & = \Theta(\beta^{\bullet}) \cdot \theta(\beta^{\bullet}(g)) \cdot \theta(C_S(P)) = \Theta(\beta^{\bullet}) \cdot \theta(C_S(P)).
\end{aligned}
\]
In particular, $\Theta_P(\alpha) \supseteq \Theta(\beta^{\bullet})$, and condition (ii) in Lemma~\ref{fmt2} also holds.
\end{proof}

Let $\hh \subseteq \FF^q$, and let $(\Gamma, \theta, \Theta)$ be a fusion mapping triple for $\FF_{\hh}$. For a subgroup $H \leq \Gamma$, let $\FF_H^{\ast} \subseteq \FF$ be the smallest \emph{restrictive} (in the sense of \cite[Definition B.6]{Gonza2}) subcategory which contains all $f \in \Mor(\FF^q)$ such that $\Theta(f) \cap H \neq \emptyset$. Let also $\FF_H \subseteq \FF_H^{\ast}$ be the full subcategory whose objects are the subgroups of $\theta^{-1}(H)$. The following result is a modification of the statement of \cite[Proposition B.8]{Gonza2} for groups of order prime to $p$ (the statement in \cite{Gonza2} dealed with $p$-groups). As the proof is exactly the same, we omit it.

\begin{proposition}\label{app3}
Let $(\Gamma, \theta, \Theta)$ be a fusion mapping triple on $\FF^q$, where $\Gamma$ is a finite group of order prime to $p$. Then the following holds for all $H \leq \Gamma$.
\begin{enumerate}[label=\textup{(\roman*)}]
\item $\FF_H$ is a saturated fusion system over $S_H = \theta^{-1}(H)$.
\item A subgroup $P \leq S_H$ is $\FF_H$-quasicentric if and only if it is $\FF$-quasicentric.
\end{enumerate}
\end{proposition}

When $\Gamma$ is a group of order prime to $p$, there is only one possible morphism from a discrete $p$-toral group $S$ to $\Gamma$, the trivial one. The  existence of a fusion mapping triple in this case is equivalent to the existence of a functor $\widehat{\Theta} \colon \FF^q \to \bb(\Gamma)$ such that $\Theta(f) = \{\widehat{\Theta}(f)\}$ for each $f \in \Mor(\FF^q)$. This equivalent approach will be useful later on when constructing fusion mapping triples.

Given a (possibly infinite) group $G$, recall that $O^{p'}(G)$ is the intersection of all normal subgroups $K \lhd G$ such that $|G/K|$ is finite and prime to $p$.

\begin{lemma}\label{opprime}
Let $G, H$ be groups, and let $f \colon G \to H$ be an epimorphism with $\Ker(f) \leq O^{p'}(G)$. Then, $f$ induces an isomorphism $\overline{f} \colon G/O^{p'}(G) \stackrel{\cong} \longrightarrow H/O^{p'}(H)$.
\end{lemma}

\begin{proof}
Let $K \lhd G$ be such that $|G/K|$ is finite and prime to $p$, and note that $\Ker(f) \leq O^{p'}(G) \leq K$ by assumption. Hence, $f$ induces an isomorphism $G/K \cong H/f(K)$. Conversely, let $N \lhd H$ be such that $|H/N|$ is finite and prime to $p$, and let $K \leq G$ be the preimage of $N$ through $f$, so $\Ker(f) \leq K$. Again, $f$ induces an isomorphism $G/K \cong H/N$. It follows that $f(O^{p'}(G)) = O^{p'}(H)$, and there is a commutative diagram of extensions
\[
\xymatrix{
O^{p'}(G) \ar[r] \ar[d]_f & G \ar[r] \ar[d]^f & G/O^{p'}(G) \ar[d]^{\overline{f}}\\
O^{p'}(H) \ar[r] & H \ar[r] & H/O^{p'}(H)
}
\]
where all the vertical maps are epimorphisms. As $\Ker(f) \leq O^{p'}(G)$, it then follows that $\overline{f}$ must be an isomorphism.
\end{proof}

Although \cite[Proposition 2.6]{BCGLO2} was originally proved for $p$-local finite groups, a careful inspection of its proof shows that it applies without modification in the compact case. Thus, combining \cite[Proposition 2.6]{BCGLO2} with Lemma~\ref{opprime} we deduce the existence of an isomorphism
\begin{equation}\label{Gamma}
\Gamma_{p'}(\FF) \definicio \pi_1(|\LL|)/O^{p'}(\pi_1(|\LL|)) \cong \pi_1(|\FF^c|)/O^{p'}(\pi_1(|\FF^c|)).
\end{equation}
We shall show that $\Gamma_{p'}(\FF)$ is a finite group of order prime to $p$, and that the natural functor
\[
\varepsilon \colon \FF^c \longrightarrow \bb (\Gamma_{p'}(\FF))
\]
induces a bijective correspondence between subgroups of $\Gamma_{p'}(\FF)$ and fusion subsystems of $\FF$ of index prime to $p$.

\begin{definition}\label{out0}
We denote by $O^{p'}_{\ast}(\FF) \subseteq \FF$ the smallest fusion subsystem over $S$ (not necessarily saturated) whose morphism set contains $O^{p'}(\Aut_{\FF}(P))$ for all $P \leq S$. Furthermore, we define $\Out^0_{\FF}(S) \leq \Out_{\FF}(S)$ as the subgroup generated by the elements $[f] \in \Out_{\FF}(S)$ such that $f|_P \in \Mor_{O^{p'}_{\ast}(\FF)}(P, S)$ for some $P \in \Ob(\FF^c)$ and $f \in \Aut_{\FF}(S)$ representing $[f]$.
\end{definition}

\begin{lemma}\label{app1}
The following holds.
\begin{enumerate}[label=\textup{(\roman*)}]
\item $O^{p'}_{\ast}(\FF)$ is normalized by $\Aut_{\FF}(S)$: for all $f \in \Mor(O^{p'}_{\ast}(\FF))$ and all $\gamma \in \Aut_{\FF}(S)$, we have
\[
\gamma \circ f \circ \gamma^{-1} \in \Mor(O^{p'}_{\ast}(\FF));
\]
\item the fusion system $\FF$ is generated by $O^{p'}_{\ast}(\FF)$ together with $\Aut_{\FF}(S)$ and
\item $\Out^0_{\FF}(S)$ is a normal subgroup of $\Out_{\FF}(S)$.
\end{enumerate}
\end{lemma}
\begin{proof}
Parts (i) and (ii) follow from \cite[Lemma 3.4]{BCGLO2}, since the same proof applies here without modification (all the properties required have the necessary counterpart for fusion systems over discrete $p$-toral groups). Part (iii) follows then from part (i).
\end{proof}

\begin{proposition}\label{app2}
There is a unique functor $\widehat{\theta} \colon \FF^c \to \bb(\Out_{\FF}(S)/\Out^0_{\FF}(S))$ with the following properties:
\begin{enumerate}[label=\textup{(\roman*)}]
\item $\widehat{\theta}(f) = [f]$ for all $f \in \Aut_{\FF}(S)$.
\item $\widehat{\theta}(f) = [\Id]$ if $f \in \Mor(O^{p'}_{\ast}(\FF)^c)$. In particular, $\widehat{\theta}$ sends inclusion morphisms to the identity.
\end{enumerate}
Furthermore, there is an isomorphism $\overline{\theta} \colon \Gamma_{p'}(\FF) \stackrel{\cong} \longrightarrow \Out_{\FF}(S)/\Out^0_{\FF}(S)$ such that $\widehat{\theta} = \bb(\overline{\theta}) \circ \varepsilon$. In particular, $\Gamma_{p'}(\FF)$ is a finite group of order prime to $p$.
\end{proposition}
\begin{proof}
By Lemma~\ref{app1} (ii), there exist an automorphism $\alpha \in \Aut_{\FF}(S)$ and a morphism $f' \in \Hom_{O^{p'}_{\ast}(\FF)^c}(\alpha(P), Q)$ such that $f
= f' \circ \alpha|_P$. Thus, if we have two such decompositions $f = f'_1 \circ (\alpha_1)|_P = f_2' \circ (\alpha_2)|_P$, then (after factoring out inclusions)
we have
\[
(\alpha_2 \circ \alpha_1^{-1})|_P = (f'_2)^{-1} \circ f_1 \in \Iso_{O^{p'}_{\ast}(\FF)^c}(\alpha_1(P), \alpha_2(P)),
\]
which implies that $\alpha_2 \circ \alpha_1^{-1} \in \Out^0_{\FF}(S)$, and we can define
\[
\widehat{\theta}(f) = [\alpha_1] = [\alpha_2] \in \Out_{\FF}(S)/\Out^0_{\FF}(S).
\]
This also proves that $\widehat{\theta}$ is well defined on morphisms and maps all objects in $\FF^c$ to the unique object of $\bb(\Out_{\FF}(S)/\Out^0_{\FF}(S))$. By Lemma~\ref{app1} (ii) again, this functor preserves compositions, and thus is well defined. Furthermore, it satisfies conditions (i) and (ii) above by construction. The uniqueness of $\widehat{\theta}$ is clear.

Let us prove then the last part of the statement. Since $\Out_{\FF}(S)/\Out^0_{\FF}(S)$ is a finite $p'$-group, the morphism $\pi_1(|\widehat{\theta}|)$ factors through a homomorphism
\[
\overline{\theta} \colon \pi_1(|\FF^c|)/ O^{p'}(\pi_1(|\FF^c|)) \longrightarrow \Out_{\FF}(S)/\Out^0_{\FF}(S),
\]
and the inclusion of $B\Aut_{\FF}(S)$ into $|\FF^c|$ (as a subcomplex with a single vertex $S$) induces then a homomorphism
\[
\tau \colon \Out_{\FF}(S) \longrightarrow \pi_1(|\FF^c|)/ O^{p'}(\pi_1(|\FF^c|)).
\]
Furthermore, $\tau$ is an epimorphism since $\FF$ is generated by $O^{p'}_{\ast}(\FF)$ and $\Aut_{\FF}(S)$, by Lemma~\ref{app1} (ii), and because every automorphism on $O^{p'}_{\ast}(\FF)$ is a composite of restrictions of automorphisms of $p$-power order.

By part (i), the composite $\overline{\theta} \circ \tau$ is the projection of $\Out_{\FF}(S)$ onto the quotient $ \Out_{\FF}(S)/\Out^0_{\FF}(S)$, 
and $\Out^0_{\FF}(S) \leq \Ker(\tau)$ by definition of $\Out_{\FF}^0(S)$. Thus $\overline{\theta}$ is an isomorphism. As $\Out_{\FF}(S)$ is a finite group of order prime to $p$, the last part of the statement follows.
\end{proof}

\begin{theorem}\label{app4}
There is a bijective correspondence between the set of subgroups $H \leq \Gamma_{p'}(\FF) = \Out_{\FF}(S)/\Out^0_{\FF}(S)$ and the set of saturated fusion subsystems $\FF_H \subseteq \FF$ of index prime to $p$. The correspondence is given by associating to $H$ the fusion system generated by $\big(\widehat{\theta}\big)^{-1}(\bb(H))$, where $\widehat{\theta}$ is the functor in Proposition~\ref{app2}.
\end{theorem}
\begin{proof}
Let $\FF_0 \subseteq \FF$ be a saturated subsystem of index prime to $p$. That is, $\FF_0$ is a saturated subsystem over $S$ which contains $O^{p'}_{\ast}(\FF)$. Then $\Out_{\FF}^0(S) \lhd \Out_{\FF_0}(S)$, and we can set
\[
H = \Out_{\FF_0}(S) / \Out^0_{\FF}(S) \leq \Gamma_{p'}(\FF).
\]
We have to show that this provides the bijection in the statement. We first show that a morphism $f \in \Mor(\FF^c)$ is in $\FF_0$ if and only if $\widehat{\theta}(f) \in H$, which in turn implies that
\[
\FF_0 = \big(\widehat{\theta}\big)^{-1}(\bb(H)).
\]
Clearly it is enough to prove this for isomorphisms in $\FF^c$.

Let $P, Q \leq S$ be $\FF$-centric, $\FF$-conjugate subgroups, and fix $f \in \Iso_{\FF}(P,Q)$. By Lemma~\ref{app1} we can write $f = f' \circ \alpha|_P$, where $\alpha \in \Aut_{\FF}(S)$ and $f' \in \Iso_{O^{p'}_{\ast}(\FF)}(P,Q)$. Then $f$ is in $\FF_0$ if and only if $\alpha|_P$ is in $\FF_0$. Also, by definition of $\widehat{\theta}$ (and also $h$), $\widehat{\theta}(f) \in H$ if and only if $\alpha \in \Aut_{\FF_0}(S)$. Thus we have to show that $\alpha|_P \in \Mor(\FF_0)$ if and only if $\alpha \in \Aut_{\FF_0}(S)$.

The case when $\alpha \in \Aut_{\FF_0}(S)$ is clear, so let us prove the converse. Note that $\alpha(P)$ is $\FF_0$-centric, and hence fully $\FF_0$-centralized. Since $\alpha|_P$ extends to an (abstract) automorphism of $S$, axiom (II) implies that it extends to some $\alpha_1 \in \Hom_{\FF_0}(N_S(P), S)$. By \cite[Proposition 2.8]{BLO3},
\[
\alpha_1 = (\alpha|_{N_S(P)}) \circ c_g
\]
for some $g \in Z(P)$, and hence $\alpha|_{N_S(P)} \in \Hom_{\FF_0}(N_S(P), S)$. Furthermore, $P \lneqq N_S(P)$ since $P \lneqq S$ by hypothesis. Applying this process repeatedly it follows that $\alpha \in \Aut_{\FF_0}(S)$.

Now, fix a subgroup $H \leq \Gamma_{p'}(\FF)$, and let $\FF_H$ be the smallest fusion system over $S$ which contains $\big(\widehat{\theta}\big)^{-1}(\bb(H))$. We show then that $\FF_H$ is a saturated fusion subsystem of $\FF$ of index prime to $p$. Let $P, Q \leq S$ be $\FF$-centric subgroups, and note that $\Hom_{\FF_H}(P,Q)$ is the set of all morphisms $f \in \Hom_{\FF}(P,Q)$ such that $\widehat{\theta}(f) \in H$. Thus, in particular $\FF_H$ contains $O^{p'}_{\ast}(\FF)$ because all morphisms in $O^{p'}_{\ast}(\FF)$ are sent by $\widehat{\theta}$ to the identity.

Define then a map $\Theta \colon \Mor(\FF^c) \to \mathfrak{Sub}(\Gamma_{p'}(\FF))$ by setting $\Theta(f) = \{\widehat{\theta}(f)\}$, that is, each image is a singleton. Let also $\theta \in \Hom(S, \Gamma_{p'}(\FF))$ be the trivial homomorphism. Then it follows that $(\Gamma_{p'}(\FF), \theta, \Theta)$ is a fusion mapping triple of $\FF^c$ which, by Lemma~\ref{app5} extends to a unique fusion mapping triple of $\FF^q$. Thus $\FF_H$ is saturated by Proposition~\ref{app3}.

By Alperin's Fusion Theorem, \cite[Theorem 3.6]{BLO3}, $\FF_H$ is the unique saturated fusion subsystem of $\FF$ with the property that a morphism $f \in \Hom_{\FF}(P,Q)$ between $\FF$-centric subgroups of $S$ lies in $\FF_H$ if and only if $\widehat{\theta}(f) \in H$. This shows that the correspondence is bijective.
\end{proof}
\begin{remark}\label{rem_minnor}
There is a minimal fusion subsystem $O^{p'}(\FF) \subseteq \FF$ of index prime to $p$, corresponding to the trivial subgroup $\{1\} \leq \Gamma_{p'}(\FF)$. By \cite[Theorem B]{LeviLibman}, $O^{p'}(\FF)$ has a unique associated centric linking system $O^{p'}(\LL)$ (up to isomorphism), and thus there is a $p$-local compact group $O^{p'}(\g) = (S, O^{p'}(\FF), O^{p'}(\LL))$.
\end{remark}
\begin{corollary}\label{minnor}
$O^{p'}(\FF)$ is a normal subsystem of $\FF$.
\end{corollary}
\begin{proof}
Since $O^{p'}(\FF)$ is a saturated fusion subsystem over $S$, conditions (N1) and (N3) in Definition~\ref{definormal} follow immediately. Also condition (N4) is immediate since $S = S \cdot C_S(S)$. Finally, condition (N2) is a consequence of the following. The fusion mapping triple $(\Gamma, \theta, \Theta)$ associated to $O^{p'}(\FF)$ corresponds to a functor $\widehat{\Theta} \colon \FF^q \to \bb(\Gamma)$ which sends the morphisms in $O^{p'}(\FF)$ to the trivial automorphism.
\end{proof}
 

\bibliographystyle{alpha}
\bibliography{Main.bib}

\newcommand{\etalchar}[1]{$^{#1}$}
\begin{thebibliography}{BCG{\etalchar{+}}07}

\bibitem[AB95]{AB}
J.~L. Alperin and Rowen~B. Bell.
\newblock {\em Groups and representations}, volume 162 of {\em Graduate Texts
  in Mathematics}.
\newblock Springer-Verlag, New York, 1995.

\bibitem[AKO11]{AKO}
Michael Aschbacher, Radha Kessar, and Bob Oliver.
\newblock {\em Fusion systems in algebra and topology}, volume 391 of {\em
  London Mathematical Society Lecture Note Series}.
\newblock Cambridge University Press, Cambridge, 2011.

\bibitem[Asc08]{Aschbacher}
Michael Aschbacher.
\newblock Normal subsystems of fusion systems.
\newblock {\em Proc. Lond. Math. Soc. (3)}, 97(1):239--271, 2008.

\bibitem[BCG{\etalchar{+}}07]{BCGLO2}
C.~Broto, N.~Castellana, J.~Grodal, R.~Levi, and B.~Oliver.
\newblock Extensions of {$p$}-local finite groups.
\newblock {\em Trans. Amer. Math. Soc.}, 359(8):3791--3858 (electronic), 2007.

\bibitem[Bla58]{Blackburn}
N.~Blackburn.
\newblock On a special class of {$p$}-groups.
\newblock {\em Acta Math.}, 100:45--92, 1958.

\bibitem[BLO03]{BLO2}
Carles Broto, Ran Levi, and Bob Oliver.
\newblock The homotopy theory of fusion systems.
\newblock {\em J. Amer. Math. Soc.}, 16(4):779--856 (electronic), 2003.

\bibitem[BLO07]{BLO3}
Carles Broto, Ran Levi, and Bob Oliver.
\newblock Discrete models for the {$p$}-local homotopy theory of compact {L}ie
  groups and {$p$}-compact groups.
\newblock {\em Geom. Topol.}, 11:315--427, 2007.

\bibitem[BLO14]{BLO6}
Carles Broto, Ran Levi, and Bob Oliver.
\newblock An algebraic model for finite loop spaces.
\newblock {\em Algebr. Geom. Topol.}, 14(5):2915--2981, 2014.

\bibitem[BM07]{BM}
Carles Broto and Jesper~M. M{\o}ller.
\newblock Chevalley {$p$}-local finite groups.
\newblock {\em Algebr. Geom. Topol.}, 7:1809--1919, 2007.

\bibitem[Cra11]{Craven}
David~A. Craven.
\newblock {\em The theory of fusion systems}, volume 131 of {\em Cambridge
  Studies in Advanced Mathematics}.
\newblock Cambridge University Press, Cambridge, 2011.
\newblock An algebraic approach.

\bibitem[DRV07]{DRV}
Antonio D{\'{\i}}az, Albert Ruiz, and Antonio Viruel.
\newblock All {$p$}-local finite groups of rank two for odd prime {$p$}.
\newblock {\em Trans. Amer. Math. Soc.}, 359(4):1725--1764 (electronic), 2007.

\bibitem[DW94]{DW0}
W.~G. Dwyer and C.~W. Wilkerson.
\newblock Homotopy fixed-point methods for {L}ie groups and finite loop spaces.
\newblock {\em Ann. of Math. (2)}, 139(2):395--442, 1994.

\bibitem[DW95]{DW2}
W.~G. Dwyer and C.~W. Wilkerson.
\newblock The center of a {$p$}-compact group.
\newblock In {\em The \v {C}ech centennial ({B}oston, {MA}, 1993)}, volume 181
  of {\em Contemp. Math.}, pages 119--157. Amer. Math. Soc., Providence, RI,
  1995.

\bibitem[Dwy96]{DwyerBCAT94}
W.~G. Dwyer.
\newblock The centralizer decomposition of {$BG$}.
\newblock In {\em Algebraic topology: new trends in localization and
  periodicity ({S}ant {F}eliu de {G}u\'\i xols, 1994)}, volume 136 of {\em
  Progr. Math.}, pages 167--184. Birkh\"auser, Basel, 1996.

\bibitem[GL16]{GL}
A.~Gonz\'alez and R.~Levi.
\newblock Automorphisms of {$p$}-local compact groups.
\newblock {\em J. Algebra}, 467:202--236, 2016.

\bibitem[Gon16a]{Gonza3}
Alex Gonz\'alez.
\newblock Finite approximations of {$p$}-local compact groups.
\newblock {\em Geom. Topol.}, 20(5):2923--2995, 2016.

\bibitem[Gon16b]{Gonza2}
Alex Gonz{\'a}lez.
\newblock The structure of {$p$}-local compact groups of rank 1.
\newblock {\em Forum Math.}, 28(2):219--253, 2016.

\bibitem[Ish01]{ishiguro}
Kenshi Ishiguro.
\newblock Toral groups and classifying spaces of {$p$}-compact groups.
\newblock In {\em Homotopy methods in algebraic topology ({B}oulder, {CO},
  1999)}, volume 271 of {\em Contemp. Math.}, pages 155--167. Amer. Math. Soc.,
  Providence, RI, 2001.

\bibitem[LL15]{LeviLibman}
Ran Levi and Assaf Libman.
\newblock Existence and uniqueness of classifying spaces for fusion systems
  over discrete {$p$}-toral groups.
\newblock {\em J. Lond. Math. Soc. (2)}, 91(1):47--70, 2015.

\bibitem[LO02]{LO}
Ran Levi and Bob Oliver.
\newblock Construction of 2-local finite groups of a type studied by {S}olomon
  and {B}enson.
\newblock {\em Geom. Topol.}, 6:917--990 (electronic), 2002.

\bibitem[Oli14]{Oliver_Ind_p}
Bob Oliver.
\newblock Simple fusion systems over {$p$}-groups with abelian subgroup of
  index {$p$}: {I}.
\newblock {\em J. Algebra}, 398:527--541, 2014.

\bibitem[Rui07]{Ruiz}
Albert Ruiz.
\newblock Exotic normal fusion subsystems of general linear groups.
\newblock {\em J. Lond. Math. Soc. (2)}, 76(1):181--196, 2007.

\end{thebibliography}


\end{document}